\documentclass[reqno,11pt]{amsart}
\usepackage{epsfig,amscd,amssymb,amsmath,amsfonts}
\usepackage{lscape}
\usepackage{amsmath}
\usepackage{graphicx}
\usepackage{amsthm,color}
\usepackage{tikz}
\usepackage{float}
\usepackage{xparse}
\usetikzlibrary{shapes.geometric}
\usetikzlibrary{shapes,positioning,intersections,quotes}
\usetikzlibrary{graphs}
\usetikzlibrary{graphs,quotes}
\usetikzlibrary{decorations.pathmorphing}

\tikzset{snake it/.style={decorate, decoration=snake}}
\tikzset{snake it/.style={decorate, decoration=snake}}

\usetikzlibrary{decorations.pathreplacing,decorations.markings,snakes}

\newtheorem{theorem}{Theorem}[section]
\newtheorem{lemma}[theorem]{Lemma}
\newtheorem{proposition}{Proposition}[section]
\theoremstyle{definition}
\newtheorem{definition}[theorem]{Definition}
\newtheorem{corollary}[theorem]{Corollary}

\newtheorem{question}[theorem]{Question}
\newtheorem{example}[theorem]{Example}

\theoremstyle{remark}
\newtheorem{remark}[theorem]{Remark}

\numberwithin{equation}{section}

\usepackage[margin=1.05in]{geometry}
\usepackage[colorlinks]{hyperref}
\usepackage{siunitx}

\NewDocumentCommand{\tens}{t_}
 {%
  \IfBooleanTF{#1}
   {\tensop}
   {\otimes}%
 }
\NewDocumentCommand{\tensop}{m}
 {%
  \mathbin{\mathop{\otimes}\displaylimits_{#1}}%
 }

\NewDocumentCommand{\tensw}{t_}
 {%
  \IfBooleanTF{#1}
   {\tensopw}
   {\wedge^{S^3}}%
 }
\NewDocumentCommand{\tensopw}{m}
 {%
  \mathbin{\mathop{\wedge^{S^3}}\displaylimits_{#1}}%
 }

\begin{document}

\title[Partitions of the complete hypergraph $K_6^3$ and a determinant like function]{Partitions of the complete hypergraph $K_6^3$\\ and a determinant like function}
\author{Steven R. Lippold}
% Address of record for the research reported here
\address{Department of Mathematics and Statistics, Bowling Green State University, Bowling Green, OH 43403}

% current address
\email{steverl@bgsu.edu}


\author{Mihai D. Staic}
\address{Department of Mathematics and Statistics, Bowling Green State University, Bowling Green, OH 43403 }
\address{Institute of Mathematics of the Romanian Academy, PO.BOX 1-764, RO-70700 Bu\-cha\-rest, Romania.}

%\curraddr{}
\email{mstaic@bgsu.edu}
%	Information for third author

%    General info

%\date{January 1, 1994 and, in revised form, June 22, 1994.}


\begin{abstract} In this paper we introduce  a determinant-like map $det^{S^3}$ and study some of its properties. For this we define a graded vector space  $\Lambda^{S^3}_V$ that has similar properties with the exterior algebra $\Lambda_V$ and the exterior GSC-operad $\Lambda^{S^2}_V$ from \cite{sta2}. When $dim(V_2)=2$ we show that $dim_k(\Lambda^{S^3}_{V_2}[6])=1$ which gives the existence  and uniqueness of $det^{S^3}$. We also give an explicit formula for $det^{S^3}$ as a sum over certain $2$-partitions of the complete hypergraph $K_6^3$.   
\end{abstract}
\maketitle

%$$\tens_{i,j}$$  $$\tensw_{1\leq i,j\leq n}$$ 

%\section*{This is an unnumbered first-level section head}

%
%%%%%%%%%%%%%%%%%%%%%%%%%%%%%%%%%%%%%%%%%%%%%%%%%%%

%%%%%%%%%%%%%%%%%%%%%%%%%%%%%%%%%%%%%%%%%%%%%%%%%%%%
\section{Introduction}

Exterior algebra $\Lambda_V$ is the quotient of the tensor algebra $T(V)$ by the ideal generated by all elements $u\otimes u$ where $u\in V$. It is well know that if $dim(V_d)=d$ then $dim_k(\Lambda_{V_d}[d])=1$, and that one gets the determinant of a linear transformation $T:V\to V$ as the unique constant $det(T)=\Lambda(T):\Lambda_{V_d}[d]\to \Lambda_{V_d}[d]$. Equivalently, the determinant  map is the unique (up to a constant) nontrivial linear function $det:V_d^{\otimes d}\to k$ with the property $det(\otimes_{1\leq i\leq d}(v_i))=0$ if there exists $1\leq x<y\leq d$ such that $v_x=v_y$. This gives the well known formula  $$det(\otimes_{1\leq i\leq d}(v_{i}))=\sum_{\sigma\in S_d}\varepsilon(\sigma)v_{1}^{\sigma(1)}\dots v_{d}^{\sigma(d)},$$ where $v_i=(v_i^1,\dots,v_i^{d})$. 

The exterior Graded-Swiss-Cheese (GSC) operad  ${\Lambda}^{S^2}_V$ was introduced in \cite{sta2} as a quotient of the tensor GSC-operad $\mathcal{T}_{V}^{S^2}$ by the ideal $\mathcal{E}_V^{S^2}$ generated by elements of the form 
$\begin{pmatrix}
1&u&u\\
&1&u\\
\otimes &&1\\
\end{pmatrix}$ where $u\in V$.  

It was proved in \cite{sta2} that $dim_k({\Lambda}^{S^2}_{V_2}[4])=1$, and in \cite{edge} that $dim_k({\Lambda}^{S^2}_{V_3}[6])=1$. In particular, if $d=2$ (or $d=3$)  we have a determinant like function $det^{S^2}:V_d^{\otimes d(2d-1)}\to k$ that is unique with the property that $det^{S^2}(\otimes_{1\leq i<j\leq 2d} (v_{i,j}))=0$ if there exists $1\leq x<y<z\leq 2d$ such that $v_{x,y}=v_{x,z}=v_{y,z}$. Moreover, when $d=2$ it was shown in \cite{edge} that $det^{S^2}:V_2^{\otimes 6}\to k$ can be written as  
$$det^{S^2}(\otimes_{1\leq i<j\leq 4}(v_{i,j}))=\sum_{(\Gamma_1,\Gamma_2)\in \mathcal{P}^{h,cf}_2(K_{4})} \varepsilon_2^{S^2}((\Gamma_1,\Gamma_2))M_{(\Gamma_1,\Gamma_2)}(\otimes_{1\leq i<j\leq 4}(v_{i,j}))$$ 
where $\mathcal{P}^{h,cf}_2(K_{4})$ is the set of all homogeneous cycle-free $2$-partitions of the complete graph $K_4$, 
$\varepsilon_2^{S^2}$ is a function defined on the set $\mathcal{P}^{h,cf}_2(K_{4})$ (similar with the signature map for  permutations), and  $M_{(\Gamma_1,\Gamma_2)}(\otimes_{1\leq i<j\leq 4}(v_{i,j}))$ is a monomial associated to  $\otimes_{1\leq i<j\leq 4}(v_{i,j})\in V_2^{\otimes 6}$. A similar formula also works for the case $d=3$.  It was conjectured in \cite{sta2} that a function $det^{S^2}:V_d^{\otimes d(2d-1)}\to k$ with the above universality property  exists and is unique for any $d$.  One can notice similarities with the exterior algebra construction and the determinant function.  

%Although the operad language is useful in giving a formal definition, one can introduce $\Lambda_{V}^{S^2}$  as a graded vector space without talking about GSC-operads. In this paper we will follow a similar  approach, avoiding any reference to the operad language. 

In this paper we consider a similar construction associated to the sphere $S^3$. The main result is the existence and uniqueness of a determinant like function $det^{S^3}:V_2^{\otimes 20}\to k$ with the property that $det^{S^3}(\otimes_{1\leq i<j<k\leq 6}(v_{i,j,k}))=0$ if there exists $1\leq x<y<z<t\leq 6$ such that $v_{x,y,z}=v_{x,y,t}=v_{x,z,t}=v_{y,z,t}$. We prove that $det^{S^3}$ is invariant under the actions of $SL_2(k)$ and of the symmetric group $S_6$. Moreover the  $det^{S^3}$ can be written as 
$$det^{S^3}(\otimes_{1\leq i<j<k\leq6}(v_{i,j,k}))=\sum_{(\mathcal{H}_1,\mathcal{H}_2)\in \mathcal{P}_2^{h,nt}(K_{6}^3)}\varepsilon^{S^3}(\mathcal{H}_1,\mathcal{H}_2)M^{S^3}_{(\mathcal{H}_1,\mathcal{H}_2)}(\otimes_{1\leq i<j<k\leq6}(v_{i,j,k})),$$
where $\mathcal{P}_2^{h,nt}(K_{6}^3)$ is a  set of homogeneous nontrivial $2$-partitions of the complete hypergraph $K_6^3$,  $\varepsilon^{S^3}$ is a function on $\mathcal{P}_2^{h,nt}(K_{6}^3)$, and $M^{S^3}_{(\mathcal{H}_1,\mathcal{H}_2)}(\otimes_{1\leq i<j<k\leq6}(v_{i,j,k}))$ is a monomial associated to $\otimes_{1\leq i<j<k\leq6}(v_{i,j,k})\in V_2^{\otimes 20}$. 

The paper is organized as follows: in section 2 we recall a few results about $\Lambda^{S^2}_V$, and some definitions and examples of hypergraphs. In section 3 we introduce $\Lambda^{S^3}_V=\oplus_{n\geq 0}\Lambda^{S^3}_V[n]$ as the quotient of $\mathcal{T}_{V}^{S^3}$ by a certain subspace generated by elements similar to $\begin{pmatrix}
 u&\otimes 	 \\
 u& u\\
      & u\\
\end{pmatrix}$ for all $u\in V$.  We also give a presentation with generators and relations using a connection with partitions of the complete hypergraph $K_n^3$. 

Section 4  deals with the case $dim(V_2)=2$. We prove the main result of this paper, namely  the existence and uniqueness of the $det^{S^3}$ map. We study some properties of the map $det^{S^3}$ and compute $dim_k(\Lambda^{S^3}_{V_2}[n])$ for all $n\geq 0$.  In section 5 we discuss  a few related  problems and possible directions of research. % We discuss how the results and ideas in this paper can be seen as a first step towards developing a linear Ramsey theory.   

Using the fact that $det^{S^3}$ is invariant under the action of $SL_2(k)$, in Appendix we give an explicit presentation for $det^{S^3}$ as a sum of  products of determinants of $2\times 2$ matrices. We also  present a classification of a certain class of homogeneous $2$-partitions of complete hypergraph $K_6^3$ under the action of the group $S_6\times S_2$ (this was obtained using MATLAB).

\section{Preliminary}

\subsection{Generalizations of the Exterior Algebra}

In this paper $k$ is an infinite field such that $char(k)\neq2$ and $char(k)\neq3$. The tensor product $\otimes$ is over the field $k$. For a vector space $V_d$ of dimension $d$  we fix a basis  $\mathcal{B}_d=\{e_1,\dots,e_d\}$. 

As a motivation for our construction we recall a few results about the determinant and the $det^{S^2}$ map. First notice that a permutation  in  $S_d$ is nothing else but an ordered
$d$-partition of the set $\{1,2,\dots, d\}$, and so the usual formula for the determinant of a matrix $A=[v_1,\dots, v_d]=[v_i^j]_{1\leq i,j\leq d}$ can be rewritten as
$$det(A)=\sum_{\sigma\in S_d}\varepsilon(\sigma)v_{1}^{\sigma(1)}v_2^{\sigma(2)}\dots v_{n}^{\sigma(n)}=\sum_{\pi\in \mathcal{P}^{h}_d(\{1,\dots, d\})}\varepsilon(\pi)M_{\pi}(v_1,\dots, v_n),$$ where $\mathcal{P}^{h}_d(\{1,\dots, d\})$ is the set of homogeneous ordered $d$-partitions of the set $\{1,\dots,d\}$, $\pi=(\{\sigma(1)\},\{\sigma(2)\}\dots,\{\sigma(d)\})$ is the $d$-partition corresponding to the permutation $\sigma$, and $M_{\pi}(v_1,\dots, v_n)$ is the corresponding monomial expression associated to $A$ and $\sigma$. We will see later in the paper how this expression of the determinant fits into more general settings.

Next recall from  \cite{edge} the definition of $\Lambda^{S^2}_V$. 
For every $n\geq0$ we denote
\[\mathcal{T}_V^{S^2}[n]:=V^{\otimes\frac{n(n-1)}{2}}.\]
 A simple tensor in $T^{S^2}_V[n]$ will be denoted by $\otimes_{1\leq i<j\leq n}(v_{i,j})$ where $v_{i,j}\in V$. 
A general element in $T^{S^2}_V[n]$ is a sum of simple tensors.
\begin{remark}
The grading that we use in this paper is different then the one from \cite{sta2}, more precisely the relation between the two gradings is  $$\mathcal{T}_V^{S^2}[n]=\mathcal{T}_V^{S^2}(n+1)=V^{\otimes\frac{n(n-1)}{2}}.$$ 
This convention is more intuitive and is consistent with the usual grading on the exterior algebra $\Lambda_V$.  
\end{remark} 

We consider $\mathcal{E}_V^{S^2}[n]$ the subspace of $V^{\otimes\frac{n(n-1)}{2}}$ that is linear generated by all the simple tensor $\otimes_{1\leq i<j\leq n}(v_{i,j})\in \mathcal{T}_V^{S^2}[n]$, with the property that there exist $1\leq x<y<z\leq n$ such that $u_{x,y}=u_{x,z}=u_{y,z}$. 
\begin{definition} \label{s2}
 Let $V$ be a $k$ vector space. We define $\Lambda^{S^2}_V$ as the graded vector space with the component in degree $n$ defined as the quotient vector space
 \[\Lambda^{S^2}_V[n]=\frac{\mathcal{T}_V^{S^2}[n]}{\mathcal{E}_V^{S^2}[n]}\]
 for every $n\geq0$.
\end{definition}
We recall a few results about  $\Lambda^{S^2}_V$. 
\begin{proposition}[\cite{sta2}, \cite{edge}]
   (i)  Let $dim(V_d)=d$. Then $dim\left(\Lambda_{V_d}^{S^2}[n]\right)=0$ for $n>2d$.\\
	(ii) If $dim(V_2)=2$ then $dim\left(\Lambda^{S^2}_{V_2}[4]\right)=1$.\\
	(iii) If $dim(V_3)=3$ then $dim\left(\Lambda^{S^2}_{V_3}[6]\right)=1$.
	\label{prop1}
\end{proposition}
\begin{remark} It was conjectured in \cite{sta2} that if $dim(V_d)=d$ then $dim\left(\Lambda^{S^2}_{V_d}[2d]\right)=1$. Notice that if the conjecture is true then we get the existence of a unique nontrivial linear map $$det^{S^2}: V_d^{\otimes d(2d-1)}\to k$$ with the property that $det^{S^2}(\otimes(v_{i,j})_{1\leq i<j\leq 2d})=0$ if there exist $1\leq x<y<z\leq 2d$ such that $v_{x,y}=v_{x,z}=v_{y,z}$. When $d=2$ or $d=3$  such a map exists and it has the expression
\begin{equation}
det^{S^2}(\otimes_{1\leq i<j\leq 2d} (v_{i,j}))=\sum_{(\Gamma_1,...,\Gamma_d)\in \mathcal{P}^{h,cf}_d(K_{2d})} \varepsilon_d^{S^2}((\Gamma_1,...,\Gamma_d))M_{(\Gamma_1,...,\Gamma_d)}(\otimes_{1\leq i<j\leq 2d}(v_{i,j})).
\label{detS2d}
\end{equation}
Here the sum is taken over all the cycle-free homogeneous $d$-partitions $(\Gamma_1,...,\Gamma_d)$ of the complete graph $K_{2d}$, and $M_{(\Gamma_1,...,\Gamma_d)}(\otimes_{1\leq i<j\leq 2d}(v_{i,j}))$ is a certain monomial associated to the partition $(\Gamma_1,...,\Gamma_d)$ and to the element $\otimes_{1\leq i<j\leq 2d}(v_{i,j})\in V_d^{\otimes d(2d-1)}$. 
\end{remark}

\begin{remark}
When $d=2$ or $d=3$ the condition of $det^{S^2}(\otimes (v_{i,j})_{1\leq i<j\leq 2d})=0$ has a nice geometrical interpretation that  was discussed in \cite{sv}. It is interesting to notice that the case $d=2$ is essentially equivalent with an old result of Pappus of Alexandria (see \cite{c}). 
\end{remark}

\begin{remark} The construction of $\Lambda^{S^2}_V$ was motivated by results from \cite{p} on higher Hochschild homology, and from \cite{vo} on Swiss-Cheese operads. The notation $det^{S^2}$ is justified by the fact that in the construction of $\Lambda^{S^2}_V$ we use a particular simplicial presentation of the sphere $S^2$.  If we extend that analogy, one can argue that the usual determinant should be denoted as $det^{S^1}$. In this paper we will deal with similar constructions associated to a simplicial structure of the sphere $S^3$. 
\end{remark}

\subsection{Partition of hypergraphs}

We recall from \cite{bretto} a few definitions and examples of hypergraphs that will be used later in this paper.  
\begin{definition} A hypergraph $\mathcal{H}=(V, E)$ consists of two finite sets  $V = \{v_1, v_2, \dots , v_n\}$ called  the set of vertices, and  $E=\{E_1, E_2, ... , E_m\}$  a family of subsets of $V$ called the hyperedges of $\mathcal{H}$. \\
If every hyperedge of $\mathcal{H}$ is of size $r$, then $\mathcal{H}$ is called an $r$-uniform hypergraph. \\
For $2 \leq r \leq n$, we define the complete
$r$-uniform hypergraph to be the hypergraph $K^r_n= (V, E)$ for which $V=\{1,2,\dots ,n\}$, and  $E$
is the family of all subsets of $V$ of size $r$.\\
\end{definition} 

A $2$-uniform hypergraph is nothing else but a graph, and $K_n^2$ is the complete graph $K_n$. 
In this paper we are interested in $3$-uniform hypergraphs.  A hyperedge of a $3$-uniform hypergraph will be called a face.

\begin{definition} Let $\mathcal{H}$ be a hypergraph and $k\geq 2$ be a natural number. A $k$-partition of $\mathcal{H}$ is an ordered collection $\mathcal{P}=(\mathcal{H}_1,\mathcal{H}_2,...,\mathcal{H}_k)$ of sub-hypergraphs $\mathcal{H}_i$ of  $\mathcal{H}$ such that:\\
(i) $V(\mathcal{H}_i)=V(\mathcal{H})$ for all $1\leq i \leq k$,  \\
(ii) $E(\mathcal{H}_i)\cap E(\mathcal{H}_j)=\emptyset$ for all $i\neq j$, \\
(iii) $\cup_{i=1}^nE(\mathcal{H}_i)=E(\mathcal{H})$. \\
We say that the partition $\mathcal{P}=(\mathcal{H}_1,\mathcal{H}_2,...,\mathcal{H}_k)$ in homogeneous if $|E(\mathcal{H}_i)\vert=|E(\mathcal{H}_j)\vert$ for all $1\leq i<j\leq k$. \\
 We will denote by $\mathcal{P}_d(\mathcal{H})$ the set of $d$-partitions of the hypergraph $\mathcal{H}$, and with $\mathcal{P}_d^{h}(\mathcal{H})$ the set of homogeneous $d$-partitions of the hypergraph $\mathcal{H}$. 
\end{definition}

Since we are only interested in  $3$-uniform hypergraphs we will draw each hyperedge as a triangle (face) that connects three vertices.  In order to avoid drawing three dimensional pictures we will draw a projection in the plane, allowing the possibility to draw the same vertex several times in our picture. 
Finally, because in this paper we are interested mostly in $2$-partitions, when we draw a partition $(\mathcal{H}_1,\mathcal{H}_2)$, we will shade the hyperedges in $\mathcal{H}_1$, and do not shade the hyperedges in $\mathcal{H}_2$. 

\begin{example} (i) Consider the complete hypergraph $K_4^3$. Take %$V(\mathcal{H}_1)=V(\mathcal{H}_2)=\{1,2,3,4\}$, 
$E(\mathcal{H}_1)=\{\{1,2,4\},\{1,3,4\}\}$, and $E(\mathcal{H}_2)=\{\{1,2,3\},\{2,3,4\}\}$, then $(\mathcal{H}_1, \mathcal{H}_2)$ is a homogeneous, $2$-partition for $K_4^3$ (see  Figure \ref{2part1}).\\
(ii) Consider the complete graph $K_4^3$. Take %$V(\mathcal{L}_1)=V(\mathcal{L}_2)=\{1,2,3,4\}$, 
$E(\mathcal{L}_1)=\{\{1,2,3\},\{1,2,4\},\{1,3,4\}\}$, and $E(\mathcal{L}_2)=\{\{2,3,4\}\}$, then $(\mathcal{L}_1, \mathcal{L}_2)$ is a  $2$-partition for $K_4^3$ that is not homogeneous (see  Figure \ref{2part2}).
\end{example}

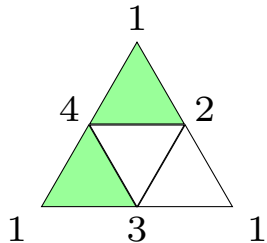
\begin{figure}[h!]
 \centering
 \begin{tikzpicture}[scale=2,every node/.style={scale=2}]
 %%%%%%%%%%%%%%%%%%%
 % this tikz is for only a's in the more horizontal alignment
  \node[regular polygon,rotate=0,draw,regular polygon sides = 3,fill=green!40] (t124) at (0,-0.72) {};
  \node[regular polygon,rotate=60,draw,regular polygon sides = 3] (t234) at (0,-1.09) {};
  \node[regular polygon,rotate=0,draw,regular polygon sides = 3,fill=green!40] (t346) at (-0.32,-1.27) {};
  \node[regular polygon,rotate=0,draw,regular polygon sides = 3] (t236) at (0.32,-1.27) {};

  \node at (-0.8,-1.6) {\tiny1};
  \node at (0.8,-1.6) {\tiny1};
  \node at (0,-0.2) {\tiny1};
  \node at (0.45,-0.8) {\tiny2};
  \node at (-0,-1.6) {\tiny3};
  \node at (-0.45,-0.8) {\tiny4};
 \end{tikzpicture}
 \caption{$(\mathcal{H}_1,\mathcal{H}_2)$ a homogeneous $2$-partition of the complete hypergraph $K_4^3$ }\label{2part1}
\end{figure}

\begin{figure}[h!]
 \centering
 \begin{tikzpicture}[scale=2,every node/.style={scale=2}]
 %%%%%%%%%%%%%%%%%%%
 % this tikz is for only a's in the more horizontal alignment
  \node[regular polygon,rotate=0,draw,regular polygon sides = 3,fill=green!40] (t124) at (0,-0.72) {};
  \node[regular polygon,rotate=60,draw,regular polygon sides = 3] (t234) at (0,-1.09) {};
  \node[regular polygon,rotate=0,draw,regular polygon sides = 3,fill=green!40] (t346) at (-0.32,-1.27) {};
  \node[regular polygon,rotate=0,draw,regular polygon sides = 3,fill=green!40] (t236) at (0.32,-1.27) {};

  \node at (-0.8,-1.6) {\tiny1};
  \node at (0.8,-1.6) {\tiny1};
  \node at (0,-0.2) {\tiny1};
  \node at (0.45,-0.8) {\tiny2};
  \node at (-0,-1.6) {\tiny3};
  \node at (-0.45,-0.8) {\tiny4};
 \end{tikzpicture}
 \caption{$(\mathcal{L}_1,\mathcal{L}_2)$ a non-homogeneous $2$-partition of the complete hypergraph $K_4^3$ }\label{2part2}
\end{figure}
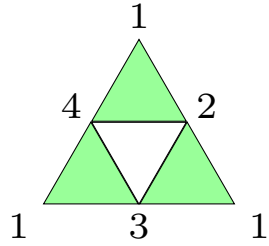

\begin{remark} Notice that if $\mathcal{H}$ is a sub-hypergraph of $K_n^3$ and $\sigma\in S_n$ then  $\sigma\cdot \mathcal{H}$ is also a sub-hypergraph of $K_n^3$, where $V(\sigma\cdot \mathcal{H})=\{\sigma(v)\vert v\in V(\mathcal{H})\}$ and 
$E(\sigma\cdot \mathcal{H})=\{\{\sigma(a),\sigma(b),\sigma(c)\}\vert \{a,b,c\}\in E(\mathcal{H})\}$. 

With this notation one can see that on $\mathcal{P}_d^{h}(K_{3d}^3)$  there is an action of the group $S_{3d}\times S_d$ given by 
$$(\sigma,\tau)\cdot (\mathcal{H}_1,\mathcal{H}_2,\dots,\mathcal{H}_d)=(\sigma\cdot\mathcal{H}_{\tau^{-1}(1)},\sigma\cdot\mathcal{H}_{\tau^{-1}(2)},\dots,\sigma\cdot\mathcal{H}_{\tau^{-1}(d)}).$$
Later in the paper we will use only the case $d=2$. 
\end{remark}

\section{Generators and Relations for $\Lambda^{S^3}_V$}

In this section, we  define $\Lambda^{S^3}_{V}$ for a finite dimensional vector space $V$ and discuss connections with partitions of hypergraphs. 
 
Take $$\mathcal{T}^{S^3}_{V}=\bigoplus_{n\geq 0}\mathcal{T}^{S^3}_{V}[n],$$ 
where $$\mathcal{T}^{S^3}_{V}[n]=V^{\otimes\frac{n(n-1)(n-2)}{6}}.$$ A simple tensor in $\mathcal{T}^{S^3}_{V}[n]$ is  denoted by $\otimes_{1\leq i<j<k\leq n} (v_{i,j,k})$ where $v_{i,j,k}\in V$. A general element in $\mathcal{T}^{S^3}_{V}[n]$ is a sum of simple tensors. One should think about $\mathcal{T}^{S^3}_{V}$ as a generalization of the tensor algebra $T(V)$, or of the tensor  GSC-operad  $\mathcal{T}^{S^2}_{V}$.

When convenient we will also use a tensor matrix notation similar with the ones from \cite{car1}, or \cite{sta2}, 
\begin{eqnarray}
\omega=\otimes_{1\leq i<j<k\leq n} (v_{i,j,k})=\begin{pmatrix}
v_{1,2,3}	 &&&\otimes \\
v_{1,2,4}&v_{1,3,4}&	&\\
         &v_{2,3,4}& &\\
v_{1,2,5}&v_{1,3,5}&v_{1,4,5}  &\\
          &v_{2,3,5}&v_{2,4,5} &\\
	     &        &v_{3,4,5}    &\\
\dots&\dots& \dots &\\
%v_{1,2,5}&v_{1,3,5}&v_{1,4,5}  &\\
 %         &v_{2,3,5}&v_{2,4,5} &\\
	%      &        &v_{3,4,5}    &\\
	  v_{1,2,6}&\dots &v_{1,n-2,n}&v_{1,n-1,n}\\
           &\dots &v_{2,n-2,n}&v_{2,n-1,n}\\
	          &        &\dots& .\\
		         &        &       &v_{n-2,n-1,n}\\
\end{pmatrix}\in V^{\otimes\frac{n(n-1)(n-2)}{6}}=\mathcal{T}^{S^3}_{V}[n].
\label{eqTn}
\end{eqnarray}
%One can think that the vector $v_{i,j,k}$  is in the  position $(i,j,k)$ of a $n\times n\times n$ cube. %Since drawing three dimensional matrices it is not very convenient, we will consider the $k$-th slice  as the set of all triple $(i,j,k)$ where $1\leq i<j<k$, and stack all these slices in one single tensor matrix. 

\begin{definition}
Take  $\mathcal{E}^{S^3}_V[n]$ to be the subspace of 
$\mathcal{T}^{S^3}_V[n]$ generated by tensors $\otimes_{1\leq i<j<k\leq n} (v_{i,j,k})$ with the property that there exists $1\leq x<y<z<t\leq n$ such that $v_{x,y,z}=v_{x,y,t}=v_{x,z,t}=v_{y,z,t}$. 
We define 
$$\Lambda^{S^3}_{V}[n]= \frac{\mathcal{T}^{S^3}_V[n]}{\mathcal{E}^{S^3}_V[n]},$$
and 
$$\Lambda^{S^3}_V=\bigoplus_{n\geq 0}\Lambda^{S^3}_V[n].$$
\end{definition}
Again, one can think of $\Lambda^{S^3}_V$ as a generalization of the exterior algebra, or of the exterior GSC-operad 
$\Lambda^{S^2}_V$. 

The image of the element  $\omega=\otimes_{1\leq i<j<k\leq n} (v_{i,j,k})\in \mathcal{T}^{S^3}_{V}[n]$ from \ref{eqTn} in $\Lambda^{S^3}_V[n]$ will be denoted as $\hat{\omega}=\wedge_{1\leq i<j<k\leq n} (v_{i,j,k})$, or as 
\begin{eqnarray*}
\hat{\omega}= \begin{pmatrix}
v_{1,2,3}	 &&&\wedge \\
v_{1,2,4}&v_{1,3,4}&	&\\
         &v_{2,3,4}& &\\
v_{1,2,5}&v_{1,3,5}&v_{1,4,5}  &\\
          &v_{2,3,5}&v_{2,4,5} &\\
	     &        &v_{3,4,5}    &\\
\dots&\dots& \dots &\\
%v_{1,2,5}&v_{1,3,5}&v_{1,4,5}  &\\
 %         &v_{2,3,5}&v_{2,4,5} &\\
	%      &        &v_{3,4,5}    &\\
	  v_{1,2,6}&\dots &v_{1,n-2,n}&v_{1,n-1,n}\\
           &\dots &v_{2,n-2,n}&v_{2,n-1,n}\\
	          &        &\dots& .\\
		         &        &       &v_{n-2,n-1,n}\\
\end{pmatrix}\in \Lambda^{S^3}_V[n]. 
\end{eqnarray*}
Let's see a few examples of identities in $\Lambda^{S^3}_{V_2}[4]$.
\begin{example}
Take $v=\alpha e_1+\beta e_2\in V_2$, we have
\begin{eqnarray*}&0=\begin{pmatrix}
 v&\wedge 	\\
v&v\\
      &v\\
\end{pmatrix}=\begin{pmatrix}
\alpha e_1+\beta e_2&\wedge 	 \\
\alpha e_1+\beta e_2&\alpha e_1+\beta e_2\\
      &\alpha e_1+\beta e_2\\
\end{pmatrix}
=\alpha^4\begin{pmatrix}
 e_1&\wedge 	 \\
 e_1& e_1\\
      & e_1\\
\end{pmatrix}+\beta^4\begin{pmatrix}
 e_2&\wedge 	 \\
 e_2& e_2\\
      & e_2\\
\end{pmatrix}&\\
&+\alpha^3\beta(\begin{pmatrix}
 e_1&\wedge 	 \\
 e_1& e_1\\
      & e_2\\
\end{pmatrix}+
\begin{pmatrix}
 e_1&\wedge 	 \\
 e_1& e_2\\
      & e_1\\
\end{pmatrix}+
\begin{pmatrix}
 e_1&\wedge 	 \\
 e_2& e_1\\
      & e_1\\
\end{pmatrix}+
\begin{pmatrix}
 e_2&\wedge 	 \\
 e_1& e_1\\
      & e_1\\
\end{pmatrix})&\\
&\alpha^2\beta^2(\begin{pmatrix}
 e_1&\wedge 	 \\
 e_1& e_2\\
      & e_2\\
\end{pmatrix}+
\begin{pmatrix}
 e_1&\wedge 	 \\
 e_2& e_1\\
      & e_2\\
\end{pmatrix}+
\begin{pmatrix}
 e_1&\wedge 	 \\
 e_2& e_2\\
      & e_1\\
\end{pmatrix}+
\begin{pmatrix}
 e_2&\wedge 	 \\
 e_1& e_1\\
      & e_2\\
\end{pmatrix}
+
\begin{pmatrix}
 e_2&\wedge 	 \\
 e_1& e_2\\
      & e_1\\
\end{pmatrix}+
\begin{pmatrix}
 e_2&\wedge 	 \\
 e_2& e_1\\
      & e_1\\
\end{pmatrix})&\\
&+\alpha\beta^3(\begin{pmatrix}
 e_1&\wedge 	 \\
 e_2& e_2\\
      & e_2\\
\end{pmatrix}+
\begin{pmatrix}
 e_2&\wedge 	 \\
 e_1& e_2\\
      & e_2\\
\end{pmatrix}+
\begin{pmatrix}
 e_2&\wedge 	 \\
 e_2& e_1\\
      & e_2\\
\end{pmatrix}+
\begin{pmatrix}
 e_2&\wedge 	 \\
 e_2& e_2\\
      & e_1\\
\end{pmatrix}).&
\end{eqnarray*}
Since this is true for all $\alpha$ and $\beta$  we get that the following identities
\begin{eqnarray}
\begin{pmatrix}
 e_1&\wedge 	 \\
 e_1& e_1\\
      & e_2\\
\end{pmatrix}+
\begin{pmatrix}
 e_1&\wedge 	 \\
 e_1& e_2\\
      & e_1\\
\end{pmatrix}+
\begin{pmatrix}
 e_1&\wedge 	 \\
 e_2& e_1\\
      & e_1\\
\end{pmatrix}+
\begin{pmatrix}
 e_2&\wedge 	 \\
 e_1& e_1\\
      & e_1\\
\end{pmatrix}=0\in \Lambda^{S^3}_{V_2}[4], 
\end{eqnarray}
 and
\begin{eqnarray}
\begin{pmatrix}
 e_1&\wedge 	 \\
 e_1& e_2\\
      & e_2\\
\end{pmatrix}+
\begin{pmatrix}
 e_1&\wedge 	 \\
 e_2& e_1\\
      & e_2\\
\end{pmatrix}+
\begin{pmatrix}
 e_1&\wedge 	 \\
 e_2& e_2\\
      & e_1\\
\end{pmatrix}+
\begin{pmatrix}
 e_2&\wedge 	 \\
 e_1& e_1\\
      & e_2\\
\end{pmatrix}
+
\begin{pmatrix}
 e_2&\wedge 	 \\
 e_1& e_2\\
      & e_1\\
\end{pmatrix}+
\begin{pmatrix}
 e_2&\wedge 	 \\
 e_2& e_1\\
      & e_1\\
\end{pmatrix}=0\in \Lambda^{S^3}_{V_2}[4].
\end{eqnarray}
\end{example}
More generally we have the following result.
\begin{proposition} \label{proprel} Let  $V_d$ be a vector space of dimension $d$, and $\mathcal{B}_d=\{e_1,\dots,e_d\}$ a basis for $V_d$. Then $\mathcal{E}^{S^3}_{V_d}[4]$ is the subspace of  $\mathcal{T}^{S^3}_{V_d}[4]$ linearly generated by the following elements: 
\begin{eqnarray}
\begin{pmatrix}
 e_i&\otimes 	 \\
 e_i& e_i\\
      & e_i\\
\end{pmatrix},\label{rel1}
\end{eqnarray}
for all $1\leq i\leq d$,

\begin{eqnarray}
\begin{pmatrix}
 e_i&\otimes 	 \\
 e_i& e_i\\
      & e_j\\
\end{pmatrix}+
\begin{pmatrix}
 e_i&\otimes 	 \\
 e_i& e_j\\
      & e_i\\
\end{pmatrix}+
\begin{pmatrix}
 e_i&\otimes 	 \\
 e_j& e_i\\
      & e_i\\
\end{pmatrix}+
\begin{pmatrix}
 e_j&\otimes 	 \\
 e_i& e_i\\
      & e_i\\
\end{pmatrix},\label{rel2}
\end{eqnarray}
for all $1\leq i\neq j\leq d,$

\begin{eqnarray}
\begin{pmatrix}
 e_i&\otimes 	 \\
 e_i& e_j\\
      & e_j\\
\end{pmatrix}+
\begin{pmatrix}
 e_i&\otimes 	 \\
 e_j& e_i\\
      & e_j\\
\end{pmatrix}+
\begin{pmatrix}
 e_i&\otimes 	 \\
 e_j& e_j\\
      & e_i\\
\end{pmatrix}+
\begin{pmatrix}
 e_j&\otimes 	 \\
 e_i& e_i\\
      & e_j\\
\end{pmatrix}
+
\begin{pmatrix}
 e_j&\otimes 	 \\
 e_i& e_j\\
      & e_i\\
\end{pmatrix}+
\begin{pmatrix}
 e_j&\otimes 	 \\
 e_j& e_i\\
      & e_i\\
\end{pmatrix},\label{rel3}
\end{eqnarray}
for all $1\leq i<j\leq d,$

\begin{eqnarray}
\sum_{\sigma \in S_2(j,k)}(\begin{pmatrix}
 e_i&\otimes 	 \\
 e_i& e_{\sigma(j)}\\
      & e_{\sigma(k)}\\
\end{pmatrix}+
\begin{pmatrix}
 e_i&\otimes 	 \\
 e_{\sigma(j)}& e_i\\
      & e_{\sigma(k)}\\
\end{pmatrix}+
\begin{pmatrix}
 e_i&\otimes 	 \\
 e_{\sigma(j)}& e_{\sigma(k)}\\
      & e_i\\
\end{pmatrix} \nonumber\\ \label{rel4}
\\
+\begin{pmatrix}
 e_{\sigma(j)}&\otimes 	 \\
 e_i& e_i\\
      & e_{\sigma(k)}\\
\end{pmatrix}
+
\begin{pmatrix}
 e_{\sigma(j)}&\otimes 	 \\
 e_i& e_{\sigma(k)}\\
      & e_i\\
\end{pmatrix}+
\begin{pmatrix}
 e_{\sigma(j)}&\otimes 	 \\
 e_{\sigma(k)}& e_i\\
      & e_i\\
\end{pmatrix}),\nonumber
\end{eqnarray}
for all $1\leq i\leq d$, $1\leq j<k\leq d$, $i\neq j$,  $i\neq k$ with the sum taken over all permutations of the set $\{j,k\}$,

\begin{eqnarray}
\sum_{\sigma\in S_4(i,j,k,l)}\begin{pmatrix}
 e_{\sigma(i)}&\otimes 	 \\
 e_{\sigma(j)}& e_{\sigma(k)}\\
      & e_{\sigma(l)}\\
\end{pmatrix}.\label{rel5}
\end{eqnarray}
for all $1\leq i<j<k<l\leq d$, where $\sigma$ runs over all permutations of the set $\{i,j,k,l\}$. 
\end{proposition}
\begin{proof}
In the definition of $\mathcal{E}^{S^3}_V[n]$ take $v=\alpha e_i+\beta e_j+\gamma e_k +\delta e_l$. Using linearly we get several terms with coefficients homogeneous monomials of total degree $4$ in $\alpha$, $\beta$, $\gamma$ and $\delta$. The expressions corresponding of $\alpha^4$, $\alpha^3\beta$, $\alpha^2\beta^2$, $\alpha^2\beta\gamma$ and $\alpha\beta\gamma\delta$ are respectively  relations \ref{rel1}, \ref{rel2}, \ref{rel3}, \ref{rel4} and \ref{rel5}. %This process can be reverse-engineered and so we get the our statement.  
\end{proof}

\begin{remark}
One can notice that if $d=2$ only relation \ref{rel1} and \ref{rel2} make sense. If $d=3$  we can add \ref{rel3} and \ref{rel4}, while for $d\geq 4$ all five relations make sense. 
\end{remark}

\begin{remark} Even if $n\geq 4$ the above relations still give a set of generators  for  $\mathcal{E}_{V_d}^{S^3}[n]$ as a vector space.  More precisely, for $n\geq 4$ and $1\leq x<y<z<t\leq n$ we can obtain an element in $\mathcal{E}_{V_d}^{S^3}[n]$ by considering a generic element $\otimes_{1\leq i<j<k\leq n}^{(x,y,z,t)}(v_{i,j,k})\in V_d^{\otimes \frac{n(n-1)(n-2)}{6}-4}$ that has all the entries $v_{i,j,k}\in \{e_1,\dots,e_d\}$, and empty spots in the positions $(x,y,z)$, $(x,y,t)$, $(x,z,t)$ and $(y,z,t)$. In order to get an element in $\mathcal{E}_{V_d}^{S^3}[n]$ one  fills the empty positions  with any of the five relations \ref{rel1}, \ref{rel2}, \ref{rel3}, \ref{rel4} and \ref{rel5}.

For example, if we take $n=6$, $(x,y,z,t)=(1,3,4,6)$, we consider a generic element that is missing entries in the positions $(1,3,4)$, $(1,3,6)$, $(1,4,6)$ and $(3,4,6)$ 
$$\begin{pmatrix}
v_{1,2,3}	 &&&\otimes \\
v_{1,2,4}&\boxed{}&	&\\
         &v_{2,3,4}& &\\
%         .&\dots& & &\\		
v_{1,2,5}&v_{1,3,5}&v_{1,4,5}  &\\
          &v_{2,3,5}&v_{2,4,5} &\\
	      &        &v_{3,4,5}&    \\
 %        .&\dots& \dots& &\\	
	  v_{1,2,6}&\boxed{}&\boxed{}&v_{1,5,6}\\
           &v_{2,3,6}&v_{2,4,6}&v_{2,5,6}\\
	          &        &\boxed{}&v_{3,5,6}\\
		         &        &       &v_{4,5,6}\\
%						         .&\dots& \dots &\dots&.\\	
\end{pmatrix}\in V_d^{\otimes 20-4},$$
and $v_{i,j,k}\in \{e_1,\dots,e_d\}$. If we use relation \ref{rel2} then we get the following equality
\begin{eqnarray*}
&&\begin{pmatrix}
v_{1,2,3}	 &&&\wedge \\
v_{1,2,4}&\boxed{e_i}&	&\\
         &v_{2,3,4}& &\\
%         .&\dots& & &\\		
v_{1,2,5}&v_{1,3,5}&v_{1,4,5}  &\\
          &v_{2,3,5}&v_{2,4,5} &\\
	      &        &v_{3,4,5}&    \\
 %        .&\dots& \dots& &\\	
	  v_{1,2,6}&\boxed{e_i}&\boxed{e_i}&v_{1,5,6}\\
           &v_{2,3,6}&v_{2,4,6}&v_{2,5,6}\\
	          &        &\boxed{e_j}&v_{3,5,6}\\
		         &        &       &v_{4,5,6}\\
%						         .&\dots& \dots &\dots&.\\	
\end{pmatrix}+
\begin{pmatrix}
v_{1,2,3}	 &&&\wedge \\
v_{1,2,4}&\boxed{e_i}&	&\\
         &v_{2,3,4}& &\\
%         .&\dots& & &\\		
v_{1,2,5}&v_{1,3,5}&v_{1,4,5}  &\\
          &v_{2,3,5}&v_{2,4,5} &\\
	      &        &v_{3,4,5}&    \\
 %        .&\dots& \dots& &\\	
	  v_{1,2,6}&\boxed{e_i}&\boxed{e_j}&v_{1,5,6}\\
           &v_{2,3,6}&v_{2,4,6}&v_{2,5,6}\\
	          &        &\boxed{e_i}&v_{3,5,6}\\
		         &        &       &v_{4,5,6}\\
%						         .&\dots& \dots &\dots&.\\	
\end{pmatrix}+\\
\\
&&\begin{pmatrix}
v_{1,2,3}	 &&&\wedge \\
v_{1,2,4}&\boxed{e_i}&	&\\
         &v_{2,3,4}& &\\
%         .&\dots& & &\\		
v_{1,2,5}&v_{1,3,5}&v_{1,4,5}  &\\
          &v_{2,3,5}&v_{2,4,5} &\\
	      &        &v_{3,4,5}&    \\
 %        .&\dots& \dots& &\\	
	  v_{1,2,6}&\boxed{e_j}&\boxed{e_i}&v_{1,5,6}\\
           &v_{2,3,6}&v_{2,4,6}&v_{2,5,6}\\
	          &        &\boxed{e_i}&v_{3,5,6}\\
		         &        &       &v_{4,5,6}\\
%						         .&\dots& \dots &\dots&.\\	
\end{pmatrix}+
\begin{pmatrix}
v_{1,2,3}	 &&&\wedge \\
v_{1,2,4}&\boxed{e_j}&	&\\
         &v_{2,3,4}& &\\
%         .&\dots& & &\\		
v_{1,2,5}&v_{1,3,5}&v_{1,4,5}  &\\
          &v_{2,3,5}&v_{2,4,5} &\\
	      &        &v_{3,4,5}&    \\
 %        .&\dots& \dots& &\\	
	  v_{1,2,6}&\boxed{e_i}&\boxed{e_i}&v_{1,5,6}\\
           &v_{2,3,6}&v_{2,4,6}&v_{2,5,6}\\
	          &        &\boxed{e_i}&v_{3,5,6}\\
		         &        &       &v_{4,5,6}\\
%						         .&\dots& \dots &\dots&.\\	
\end{pmatrix}= 0\in \Lambda^{S^3}_{V_d}[6]
\end{eqnarray*}
for all $1\leq i\neq j\leq d$. One can easily see that in this way we get all the relations in $\Lambda^{S^3}_{V_d}[n]$. 
\end{remark}

%\subsection{Generators and hypergraph partitions}

Next we exhibit a system of generators for $\Lambda^{S^3}_{V_d}[n]$ that is indexed by $d$-partitions of the complete hypergraph $K_n^3$. This is similar to the result from   \cite{edge} which gives a relation between a set of  generators for $\Lambda^{S^2}_{V_d}[n]$ and edge partitions of $K_n$ 

Let $\mathcal{B}_d=\{e_1,\dots,e_d\}$ be a basis for $V_d$, we define
\[\mathcal{G}_{\mathcal{B}_d}^{S^3}[n]=\{\otimes_{1\leq i<j<k\leq n}(v_{i,j,k})\in\mathcal{T}^{S^3}_{V_d}[n]\;\vert \; v_{i,j,k}\in  \mathcal{B}_d\}.\]
Because of linearity it is obvious that $\mathcal{G}_{\mathcal{B}_d}^{S^3}[n]$ is a basis for 
$\mathcal{T}^{S^3}_{V_d}[n]$, and so its image in $\Lambda^{S^3}_{V_d}[n]$ will be a system of generators. 

One can notice that there exists a bijection between the elements in $\mathcal{G}_{\mathcal{B}_d}^{S^3}[n]$ and $d$-partitions of the hypergraph $K_n^3$. Indeed to every element in $\omega=\tens_{1\leq i<j<k\leq n}(v_{i,j,k})\in \mathcal{G}_{\mathcal{B}_d}^{S^3}[n]$ we associate the partition $\mathcal{P}_{\omega}=(\mathcal{H}_1,\dots,\mathcal{H}_d)$ where the hyperedge $\{a,b,c\}\in \mathcal{H}_i$ if and only if $v_{a,b,c}=e_i$. It is easy to see that this map is a bijection between
$\mathcal{G}_{\mathcal{B}_d}^{S^3}[n]$ and $\mathcal{P}_d(K_n^3)$.  For a $d$-partition $\mathcal{P}\in \mathcal{P}_d(K_n^3)$ the corresponding element in $\mathcal{G}_{\mathcal{B}_d}^{S^3}[n]$ will be denoted by $\omega_{\mathcal{P}}$.

\begin{example}
Consider the  element:
$$\omega=\begin{pmatrix}
e_1	 &&&\otimes \\
e_1&e_1&	&\\
         &e_2& &\\
e_1&e_2&e_2  &\\
          &e_2&e_2 &\\
	      &        &e_1    &\\
	  e_2&e_2&e_2&e_1\\
           &e_2&e_1&e_1\\
	          &        &e_1&e_2\\
		         &        &       &e_1\\
\end{pmatrix}\in \mathcal{G}_{\mathcal{B}_2}^{S^3}[6].$$

The corresponding partition $\mathcal{P}_{\omega}$ is presented in Figure \ref{figex1}. 
\begin{figure}[h!]
 \centering
 \begin{tikzpicture}[scale=2,every node/.style={scale=2}]
 %%%%%%%%%%%%%%%%%%%
 % this tikz is for only a's in the more horizontal alignment
  \node[regular polygon,rotate=60,draw,regular polygon sides = 3] (t126) at (0,0) {};
  \node[regular polygon,rotate=0,draw,regular polygon sides = 3,fill=green!40] (t123) at (-0.32,-0.18) {};
  \node[regular polygon,rotate=0,draw,regular polygon sides = 3,fill=green!40] (t156) at (0.32,-0.18) {};
  \node[regular polygon,rotate=0,draw,regular polygon sides = 3,fill=green!40] (t124) at (0,-0.72) {};
  \node[regular polygon,rotate=60,draw,regular polygon sides = 3,fill=green!40] (t134) at (-0.32,-0.54) {};
  \node[regular polygon,rotate=60,draw,regular polygon sides = 3,fill=green!40] (t125) at (0.32,-0.54) {};
  \node[regular polygon,rotate=0,draw,regular polygon sides = 3] (t245) at (0.63,-0.72) {};
  \node[regular polygon,rotate=60,draw,regular polygon sides = 3] (t145) at (0.94,-0.54) {};
  \node[regular polygon,rotate=0,draw,regular polygon sides = 3,fill=green!40] (t345) at (-0.63,-0.72) {};
  \node[regular polygon,rotate=60,draw,regular polygon sides = 3] (t235) at (-0.94,-0.54) {};
  \node[regular polygon,rotate=60,draw,regular polygon sides = 3] (t234) at (0,-1.09) {};
  \node[regular polygon,rotate=0,draw,regular polygon sides = 3,fill=green!40] (t346) at (-0.32,-1.27) {};
  \node[regular polygon,rotate=0,draw,regular polygon sides = 3] (t236) at (0.32,-1.27) {};
  \node[regular polygon,rotate=60,draw,regular polygon sides = 3,fill=green!40] (t246) at (0.63,-1.09) {};
  \node[regular polygon,rotate=0,draw,regular polygon sides = 3] (t146) at (0.94,-1.27) {};
  \node[regular polygon,rotate=60,draw,regular polygon sides = 3,fill=green!40] (t456) at (-0.63,-1.09) {};
  \node[regular polygon,rotate=0,draw,regular polygon sides = 3,fill=green!40] (t256) at (-0.94,-1.27) {};
  \node[regular polygon,rotate=0,draw,regular polygon sides = 3] (t135) at (0,-1.818) {};
  \node[regular polygon,rotate=60,draw,regular polygon sides = 3] (t356) at (-0.32,-1.63) {};
  \node[regular polygon,rotate=60,draw,regular polygon sides = 3] (t136) at (0.32,-1.63) {};
  % nodes for labels
  \node at (0.3,0.3) {\tiny6};
  \node at (-0.3,0.3) {\tiny2};
  \node at (-0.7,-0.25) {\tiny3};
  \node at (0.7,-0.25) {\tiny5};
  \node at (1.3,-0.25) {\tiny1};
  \node at (-1.35,-0.25) {\tiny2};
  \node at (-1.1,-0.95) {\tiny5};
  \node at (1.1,-0.95) {\tiny4};
  \node at (-1.35,-1.4) {\tiny2};
  \node at (1.35,-1.4) {\tiny1};
  \node at (-0.7,-1.6) {\tiny6};
  \node at (0.7,-1.6) {\tiny6};
  \node at (0.35,-2.1) {\tiny1};
  \node at (-0.35,-2.1) {\tiny5};
  
  % this nodes should be omitted when shading
  \node at (0,-0.5) {\tiny1};
  \node at (0.32,-1.1) {\tiny2};
  \node at (-0,-1.7) {\tiny3};
  \node at (-0.32,-1.1) {\tiny4};
 \end{tikzpicture}
 \caption{$\mathcal{P}^{(1)}=\mathcal{P}_{\omega}$ the $2$-partition of $K_6^3$ associate to $\omega$} \label{figex1}
\end{figure}
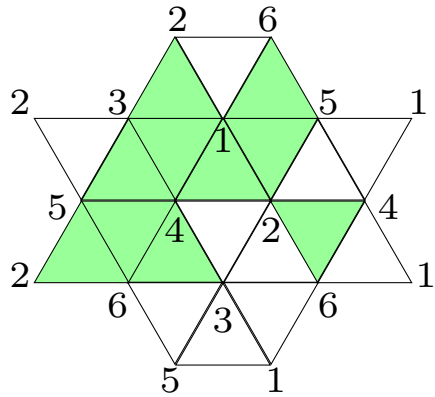

\end{example}

%In the rest of the section we will restrict to the case $d=2$ and $n=6$. 

\begin{example} Using the above dictionary between $2$-partitions of $K_n^3$ and elements from $\mathcal{G}_{\mathcal{B}_2}^{S^3}[n]$, one can translate relations \ref{rel2} and \ref{rel3}  into relations among partitions  as in  Figure \ref{figR2}, respectively Figure \ref{figR3}. 
\begin{figure}[h!]
 \centering
 \begin{tikzpicture}[scale=1.5,every node/.style={scale=1.5}]
 %%%%%%%%%%%%%%%%%%%
 % this tikz is for only a's in the more horizontal alignment
  \node[regular polygon,rotate=0,draw,regular polygon sides = 3,fill=green!40] (t124) at (0,-0.72) {};
  \node[regular polygon,rotate=60,draw,regular polygon sides = 3] (t234) at (0,-1.09) {};
  \node[regular polygon,rotate=0,draw,regular polygon sides = 3,fill=green!40] (t346) at (-0.32,-1.27) {};
  \node[regular polygon,rotate=0,draw,regular polygon sides = 3,fill=green!40] (t236) at (0.32,-1.27) {};
	\node[shape=circle,minimum size = 5pt,inner sep=0.2pt] (n8) at (1,-0.9) {{\small +}};

  \node at (-0.8,-1.6) {\tiny1};
  \node at (0.8,-1.6) {\tiny1};
  \node at (0,-0.2) {\tiny1};
  \node at (0.45,-0.8) {\tiny2};
  \node at (-0,-1.6) {\tiny3};
  \node at (-0.45,-0.8) {\tiny4};
	
	\node[regular polygon,rotate=0,draw,regular polygon sides = 3,fill=green!40] (t124) at (2,-0.72) {};
  \node[regular polygon,rotate=60,draw,regular polygon sides = 3,fill=green!40] (t234) at (2,-1.09) {};
  \node[regular polygon,rotate=0,draw,regular polygon sides = 3] (t346) at (1.68,-1.27) {};
  \node[regular polygon,rotate=0,draw,regular polygon sides = 3,fill=green!40] (t236) at (2.32,-1.27) {};
\node[shape=circle,minimum size = 5pt,inner sep=0.2pt] (n8) at (3,-0.9) {{\small +}};
	
  \node at (1.2,-1.6) {\tiny1};
  \node at (2.8,-1.6) {\tiny1};
  \node at (2,-0.2) {\tiny1};
  \node at (2.45,-0.8) {\tiny2};
  \node at (2,-1.6) {\tiny3};
  \node at (1.55,-0.8) {\tiny4};
	
		\node[regular polygon,rotate=0,draw,regular polygon sides = 3] (t124) at (4,-0.72) {};
  \node[regular polygon,rotate=60,draw,regular polygon sides = 3,fill=green!40] (t234) at (4,-1.09) {};
  \node[regular polygon,rotate=0,draw,regular polygon sides = 3,fill=green!40] (t346) at (3.68,-1.27) {};
  \node[regular polygon,rotate=0,draw,regular polygon sides = 3,fill=green!40] (t236) at (4.32,-1.27) {};
\node[shape=circle,minimum size = 5pt,inner sep=0.2pt] (n8) at (5,-0.9) {{\small +}};
	
  \node at (3.2,-1.6) {\tiny1};
  \node at (4.8,-1.6) {\tiny1};
  \node at (4,-0.2) {\tiny1};
  \node at (4.45,-0.8) {\tiny2};
  \node at (4,-1.6) {\tiny3};
  \node at (3.55,-0.8) {\tiny4};
	
	\node[regular polygon,rotate=0,draw,regular polygon sides = 3,fill=green!40] (t124) at (6,-0.72) {};
  \node[regular polygon,rotate=60,draw,regular polygon sides = 3,fill=green!40] (t234) at (6,-1.09) {};
  \node[regular polygon,rotate=0,draw,regular polygon sides = 3,fill=green!40] (t346) at (5.68,-1.27) {};
  \node[regular polygon,rotate=0,draw,regular polygon sides = 3] (t236) at (6.32,-1.27) {};
\node[shape=circle,minimum size = 5pt,inner sep=0.2pt] (n8) at (7,-0.9) {{\small=}};
	
  \node at (5.2,-1.6) {\tiny1};
  \node at (6.8,-1.6) {\tiny1};
  \node at (6,-0.2) {\tiny1};
  \node at (6.45,-0.8) {\tiny2};
  \node at (6,-1.6) {\tiny3};
  \node at (5.55,-0.8) {\tiny4};
	\node[shape=circle,minimum size = 5pt,inner sep=0.2pt] (n8) at (7.5,-0.9) {{\small 0}};
 \end{tikzpicture}
 \caption{Combinatorial representation for Equation \ref{rel2} }\label{figR2}
\end{figure}
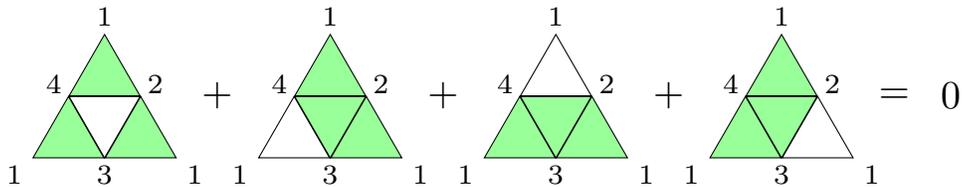

\begin{figure}[h!!]
 \centering
 \begin{tikzpicture}[scale=1.5,every node/.style={scale=1.5}]
 %%%%%%%%%%%%%%%%%%%
 % this tikz is for only a's in the more horizontal alignment
  \node[regular polygon,rotate=0,draw,regular polygon sides = 3,fill=green!40] (t124) at (0,-0.72) {};
  \node[regular polygon,rotate=60,draw,regular polygon sides = 3] (t234) at (0,-1.09) {};
  \node[regular polygon,rotate=0,draw,regular polygon sides = 3] (t346) at (-0.32,-1.27) {};
  \node[regular polygon,rotate=0,draw,regular polygon sides = 3,fill=green!40] (t236) at (0.32,-1.27) {};
	\node[shape=circle,minimum size = 5pt,inner sep=0.2pt] (n8) at (1,-0.9) {{\small +}};

  \node at (-0.8,-1.6) {\tiny1};
  \node at (0.8,-1.6) {\tiny1};
  \node at (0,-0.2) {\tiny1};
  \node at (0.45,-0.8) {\tiny2};
  \node at (-0,-1.6) {\tiny3};
  \node at (-0.45,-0.8) {\tiny4};
	
	\node[regular polygon,rotate=0,draw,regular polygon sides = 3] (t124) at (2,-0.72) {};
  \node[regular polygon,rotate=60,draw,regular polygon sides = 3] (t234) at (2,-1.09) {};
  \node[regular polygon,rotate=0,draw,regular polygon sides = 3,fill=green!40] (t346) at (1.68,-1.27) {};
  \node[regular polygon,rotate=0,draw,regular polygon sides = 3,fill=green!40] (t236) at (2.32,-1.27) {};
\node[shape=circle,minimum size = 5pt,inner sep=0.2pt] (n8) at (3,-0.9) {{\small +}};
	
  \node at (1.2,-1.6) {\tiny1};
  \node at (2.8,-1.6) {\tiny1};
  \node at (2,-0.2) {\tiny1};
  \node at (2.45,-0.8) {\tiny2};
  \node at (2,-1.6) {\tiny3};
  \node at (1.55,-0.8) {\tiny4};
	
		\node[regular polygon,rotate=0,draw,regular polygon sides = 3] (t124) at (4,-0.72) {};
  \node[regular polygon,rotate=60,draw,regular polygon sides = 3,fill=green!40] (t234) at (4,-1.09) {};
  \node[regular polygon,rotate=0,draw,regular polygon sides = 3] (t346) at (3.68,-1.27) {};
  \node[regular polygon,rotate=0,draw,regular polygon sides = 3,fill=green!40] (t236) at (4.32,-1.27) {};
\node[shape=circle,minimum size = 5pt,inner sep=0.2pt] (n8) at (5,-0.9) {{\small +}};
	
  \node at (3.2,-1.6) {\tiny1};
  \node at (4.8,-1.6) {\tiny1};
  \node at (4,-0.2) {\tiny1};
  \node at (4.45,-0.8) {\tiny2};
  \node at (4,-1.6) {\tiny3};
  \node at (3.55,-0.8) {\tiny4};
	
  \node[regular polygon,rotate=0,draw,regular polygon sides = 3,fill=green!40] (t124) at (0,-2.72) {};
  \node[regular polygon,rotate=60,draw,regular polygon sides = 3] (t234) at (0,-3.09) {};
  \node[regular polygon,rotate=0,draw,regular polygon sides = 3,fill=green!40] (t346) at (-0.32,-3.27) {};
  \node[regular polygon,rotate=0,draw,regular polygon sides = 3] (t236) at (0.32,-3.27) {};
	\node[shape=circle,minimum size = 5pt,inner sep=0.2pt] (n8) at (1,-2.9) {{\small +}};
	\node[shape=circle,minimum size = 5pt,inner sep=0.2pt] (n8) at (-1,-2.9) {{\small +}};
  \node at (-0.8,-3.6) {\tiny1};
  \node at (0.8,-3.6) {\tiny1};
  \node at (0,-2.2) {\tiny1};
  \node at (0.45,-2.8) {\tiny2};
  \node at (-0,-3.6) {\tiny3};
  \node at (-0.45,-2.8) {\tiny4};
	
	\node[regular polygon,rotate=0,draw,regular polygon sides = 3,fill=green!40] (t124) at (2,-2.72) {};
  \node[regular polygon,rotate=60,draw,regular polygon sides = 3,fill=green!40] (t234) at (2,-3.09) {};
  \node[regular polygon,rotate=0,draw,regular polygon sides = 3] (t346) at (1.68,-3.27) {};
  \node[regular polygon,rotate=0,draw,regular polygon sides = 3] (t236) at (2.32,-3.27) {};
\node[shape=circle,minimum size = 5pt,inner sep=0.2pt] (n8) at (3,-2.9) {{\small +}};
	
  \node at (1.2,-3.6) {\tiny1};
  \node at (2.8,-3.6) {\tiny1};
  \node at (2,-2.2) {\tiny1};
  \node at (2.45,-2.8) {\tiny2};
  \node at (2,-3.6) {\tiny3};
  \node at (1.55,-2.8) {\tiny4};
	
		\node[regular polygon,rotate=0,draw,regular polygon sides = 3] (t124) at (4,-2.72) {};
  \node[regular polygon,rotate=60,draw,regular polygon sides = 3,fill=green!40] (t234) at (4,-3.09) {};
  \node[regular polygon,rotate=0,draw,regular polygon sides = 3,fill=green!40] (t346) at (3.68,-3.27) {};
  \node[regular polygon,rotate=0,draw,regular polygon sides = 3] (t236) at (4.32,-3.27) {};
\node[shape=circle,minimum size = 5pt,inner sep=0.2pt] (n8) at (5,-2.9) {{\small =}};
\node[shape=circle,minimum size = 5pt,inner sep=0.2pt] (n8) at (5.5,-2.9) {{\small 0}};
  \node at (3.2,-3.6) {\tiny1};
  \node at (4.8,-3.6) {\tiny1};
  \node at (4,-2.2) {\tiny1};
  \node at (4.45,-2.8) {\tiny2};
  \node at (4,-3.6) {\tiny3};
  \node at (3.55,-2.8) {\tiny4};
	
 \end{tikzpicture}
 \caption{Combinatorial representation for Equation \ref{rel3} }\label{figR3}
\end{figure}
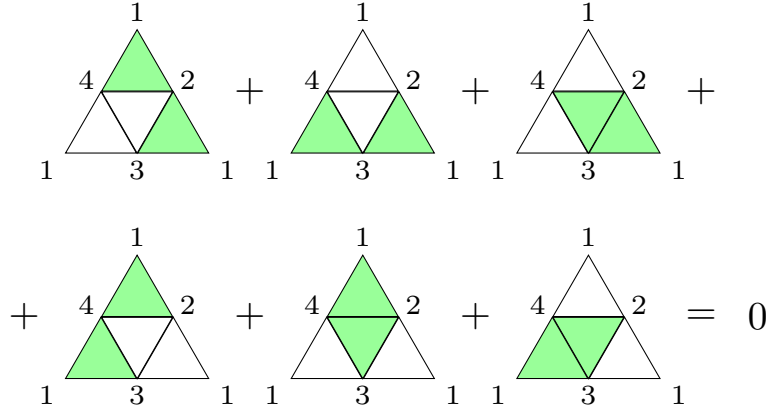
If $d\geq 3$ there are similar pictures for \ref{rel4}, \ref{rel5} and \ref{rel5}, but since in this paper we are only interested in the case $d=2$, we only present these two pictures. 
\end{example} 

\begin{remark} Notice that a more natural indexing set for the $\binom{n}{3}$ positions of a tensor in $V^{\otimes \binom{n}{3}}$ (i.e. $(i,j,k)$ with $1\leq i<j<k\leq n$),  is the set of set of hyperedges of $K_n^3$ (i.e. $\{i,j,k\}$ where $1\leq i,j,k\leq n$ with $i\neq j\neq k\neq i$). However, instead of using the notation $v_{\{i,j,k\}}$ we will use the convention $v_{i,j,k}=v_{i,k,j}=v_{j,i,k}=v_{k,i,j}=v_{j,k,i}=v_{k,j,i}$ for all $1\leq i<j<k\leq n$. 

\end{remark}

\section{Main Result} 

In this section we consider  the case $dim_k(V_2)=2$ and compute the dimension of $\Lambda^{S^3}_{V_2}[n]$ for all $n\geq 0$. In particular we show that $dim_k(\Lambda^{S^3}_{V_2}[6])=1$, which implies the existence and uniqueness of a determinant-like function $det^{S^3}$. We give an explicit description for $det^{S^3}$ using homogeneous $2$-partitions of the complete hypergraph $K_6^3$. 

First we need a technical result.
\begin{lemma}\label{trivEntries} Let $n\geq 4$ and take $\omega=\otimes_{1\leq i<j<k\leq n} (v_{i,j,k})\in \mathcal{G}_{\mathcal{B}_2}^{S^3}[n]\subseteq V_2^{\otimes \binom{n}{3}}$ (i.e. $v_{i,j,k}\in \{e_1,e_2\}$). Assume that there are at least $\binom{n-1}{2}+1$ entries equal to $e_1$ amongst the vectors $v_{i,j,k}$,   then $\hat{\omega}=0\in \Lambda^{S^3}_{V_2}[n]$. \label{lemma1}
\end{lemma}
\begin{proof}  
We will use induction. When $n=4$, we have that $\binom{4-1}{2}+1=4$ and so all entries in $\omega$ are equal to $e_1$, which means that $\omega\in \mathcal{E}^{S^3}_{V_2}[4]$, and so  $\hat{\omega}=0\in  \Lambda^{S^3}_{V_2}[4]$.

 The plan is to show that we can change $\omega$ with a sum of elements 
$\omega_p\in \mathcal{G}_{\mathcal{B}_d}^{S^3}[n]$,  such that each $\omega_p$ is either zero in $\Lambda^{S^3}_{V_2}[n]$, or it has at most $n-2$ entries equal to $e_1$ in the $n^{th}$ slice (i.e. among the entries $\{v_{i,j,n}\}_{1\leq i< j\leq n-1}$). And so, each of these elements $\omega_p$ has at least $\binom{n-1}{2}+1-(n-2)=\binom{n-2}{2}+1$ entries equal to $e_1$ in the first $n-1$ slices of $\omega_p$. By induction, we would get that $\hat{\omega}=0\in \Lambda^{S^3}_{V_2}[n]$.

Consider the $n^{th}$-slice of $\omega$, which consists of all the entries in the positions $(i,j,n)$ where $1\leq i<j\leq n-1$.  First notice that if there are more then $n-2$ entries equal to $e_1$ in the $n^{th}$ slice, then there exists a cycle $(i_1,i_2,\dots, i_q)$ where $1\leq q\leq n-1$, $1\leq i_s\leq n-1$, and $i_s\neq i_t$ for $s\neq t$ such that  $v_{i_1,i_2,n}=v_{i_2,i_3,n}=\dots v_{i_{q-1},i_q,n}=v_{i_q,i_1,n}=e_1$. In such a situation we will say that the $n^{th}$ slice has an $e_1$-cycle of length $q$. We will do a second induction over $q$. 

Assume that $n>4$ and $q=3$ then we have $v_{i_1,i_2,n}=v_{i_2,i_3,n}=v_{i_3,i_1,n}=e_1$ for some distinct integers $1\leq i_1,i_2,i_3\leq n-1$ (i.e. we have an $e_1$-cycle of length $3$). If $v_{i_1,i_2,i_3}=e_1$ then $\omega\in \mathcal{E}^{S^3}_{V_2}[n]$ and so $\hat{\omega}=0$. If $v_{i_1,i_2,i_3}=e_2$ then  using identity \ref{rel2} on the positions $(i_1,i_2,i_3)$, $(i_1,i_2,n)$, $(i_1,i_3,n)$ and $(i_2,i_3,n)$ we can move one of the $e_1$ entries from the $n^{th}$ slice to a lower slice which decrease the numbers of entries equal to $e_1$ in the $n^{th}$ slice. More precisely, looking only at the entries $ \begin{pmatrix}
 v_{i_1,i_2,i_3}&\otimes 	 \\
v_{i_1,i_2,n}& v_{i_1,i_3,n}\\
      & v_{i_2,i_3,n}\\
\end{pmatrix}$,  we have the identity
\begin{eqnarray*}
\begin{pmatrix}
 e_1&\wedge	 \\
 e_1& e_1\\
      & e_2\\
\end{pmatrix}+
\begin{pmatrix}
 e_1&\wedge		 \\
 e_1& e_2\\
      & e_1\\
\end{pmatrix}+
\begin{pmatrix}
 e_1&\wedge	 	 \\
 e_2& e_1\\
      & e_1\\
\end{pmatrix}+
\boxed{\begin{pmatrix}
 e_2&\wedge		 \\
 e_1& e_1\\
      & e_1\\
\end{pmatrix}}
=0, 
\end{eqnarray*}
where the boxed tensor matrix corresponds to our initial element. Notice that all the other tensor matrix in the above expression have the entry in the position $(i_1,i_2,i_3)$ equal to $e_1$, and so less entries equal to $e_1$ in the $n^{th}$ slice.

Next suppose that $q=4$, i.e. we have an $e_1$-cycle $(i_1,i_2,i_3,i_4)$ of length $4$, which means that  $v_{i_1,i_2,n}=v_{i_2,i_3,n}=v_{i_3,i_4,n}=v_{i_4,i_1,n}=e_1$, and $v_{i_1,i_3,n}=v_{i_2,i_4,n}=e_2$ (if any of these two entries is $e_1$ then we have a shorter $e_1$-cycle in the $n^{th}$ slice). If $v_{i_1,i_2,i_3}=v_{i_2,i_3,i_4}=v_{i_3,i_4,i_1}=v_{i_4,i_1,i_2}=e_1$,  then $\omega\in \mathcal{E}^{S^3}_{V_2}[n]$ and so $\hat{\omega}$ is trivial. 

So, we can assume that one of the  entries $v_{i_1,i_2,i_3}$, $v_{i_2,i_3,i_4}$, $v_{i_3,i_4,i_1}$ or $v_{i_4,i_1,i_2}$ is equal to $e_2$.  Then we will use identity \ref{rel3} to either move one of the $e_1$ entries from the $n^{th}$ slice to a lower slice, or to get a $e_1$-cycle of length $3$ inside the $n^{th}$ slice (and so we reduce our problem  to the case $q=3$). Indeed, for example if $v_{i_1,i_2,i_3}=e_2$  then looking only at the entries $ \begin{pmatrix}
 v_{i_1,i_2,i_3}&\otimes 	 \\
v_{i_1,i_2,n}& v_{i_1,i_3,n}\\
      & v_{i_2,i_3,n}\\
\end{pmatrix}$, from equation \ref{rel3} we have the following identity
\begin{eqnarray*}
\begin{pmatrix}
 e_1&\wedge	 \\
 e_1& e_2\\
      & e_2\\
\end{pmatrix}+
\begin{pmatrix}
 e_1&\wedge		 \\
 e_2& e_1\\
      & e_2\\
\end{pmatrix}+
\begin{pmatrix}
 e_1&\wedge	 	 \\
 e_2& e_2\\
      & e_1\\
\end{pmatrix}+
\begin{pmatrix}
 e_2&\wedge		 \\
 e_1& e_1\\
      & e_2\\
\end{pmatrix}
+
\boxed{\begin{pmatrix}
 e_2&\wedge		 \\
 e_1& e_2\\
      & e_1\\
\end{pmatrix}}+
\begin{pmatrix}
 e_2&\wedge	 	 \\
 e_2& e_1\\
      & e_1\\
\end{pmatrix}=0, 
\end{eqnarray*}
where the boxed tensor matrix corresponds to our initial element. Notice that all the other tensor matrix in the above expression have either the entry in the position $(i_1,i_2,i_3)$ equal to $e_1$ (and so less entries equal to $e_1$ in the $n^{th}$ slice), or the entry in the position $(i_1,i_3,n)$ equal to $e_1$ (and so the $e_1$-cycle $(i_1,i_3,i_4)$ of length $3$ in the $n^{th}$ slice). The other cases are similar. 

Finally, suppose that we have an  $e_1$-cycle $(i_1,i_2,...,i_q)$ in the $n^{th}$ slice with $q\geq 5$, which means that 
 $v_{i_1,i_2,n}=v_{i_2,i_3,n}=v_{i_3,i_4,n}=v_{i_4,i_5,n}=...=v_{i_{q},i_1,n}=e_1$ and $v_{i_1,i_3,n}=v_{i_2,i_4,n}=v_{i_4,i_1,n}=e_2$ (if any of these three are equal to $e_1$ then there is a shorter $e_1$-cycle in the $n^{th}$ slice). 

If $v_{i_1,i_2,i_3}=e_2$ or $v_{i_2,i_3,i_4}=e_2$ then, just like above, we can use relation \ref{rel3} to either move an $e_1$ to a lower slice, or to get an $e_1$-cycle of length at most $q-1$ in the $n^{th}$ slice. So, we can assume that  $v_{i_1,i_2,i_3}=v_{i_2,i_3,i_4}=e_1$. 

If $v_{i_1,i_3,i_4}=e_2$ then we can use relation \ref{rel2} to either move an $e_1$ to a lower slice, or to get a shorter $e_1$-cycle in the $n^{th}$ slice. Indeed, looking only at the entries $ \begin{pmatrix}
 v_{i_1,i_3,i_4}&\otimes 	 \\
v_{i_1,i_3,n}& v_{i_1,i_4,n}\\
      & v_{i_3,i_4,n}\\
\end{pmatrix}$, 
we have the following identity
\begin{eqnarray*}
\boxed{\begin{pmatrix}
 e_2&\wedge	 	 \\
 e_2& e_2\\
      & e_1\\
\end{pmatrix}}+
\begin{pmatrix}
 e_2&\wedge		 \\
 e_2& e_1\\
      & e_2\\
\end{pmatrix}+
\begin{pmatrix}
 e_2&\wedge		 \\
 e_1& e_2\\
      & e_2\\
\end{pmatrix}+
\begin{pmatrix}
 e_1&\wedge	 	 \\
 e_2& e_2\\
      & e_2\\
\end{pmatrix}=0,
\end{eqnarray*}
where the boxed tensor matrix corresponds to our initial element. Notice that all the other tensor matrix in the above expression have either the entry in the position $(i_1,i_3,i_4)$  equal to $e_1$ (and so less entries equal to $e_1$ in the $n^{th}$ slice), or the entry in the position $(i_1,i_3,n)$ or $(i_1,i_4,n)$ equal to $e_1$ (and so a shorter $e_1$-cycle  inside the $n^{th}$ slice). The  case $v_{i_1,i_2,i_4}=e_2$ is similar.  

To summarize, we  have that  $v_{i_1,i_2,i_3}=v_{i_1,i_2,i_4}=v_{i_1,i_3,i_4}=v_{i_2,i_3,i_4}=e_1$ and so $\hat{\omega}=0$. 
By induction we get our claim. 
\end{proof}
The following is the $S^3$-version of first statement of Proposition \ref{prop1}. 
\begin{corollary} If $dim_k(V_2)=2$ then $dim_k(\Lambda^{S^3}_{V_2}[n])=0$ for all $n\geq 7$. \label{cor1}
\end{corollary}
\begin{proof}  Take $n\geq 7$ and consider $\omega=\otimes_{1\leq i<j<k\leq n} (v_{i,j,k})\in \mathcal{G}_{\mathcal{B}_d}^{S^3}[n]\subseteq V_2^{\otimes \binom{n}{3}}$ (i.e. $v_{i,j,k}\in \{e_1,e_2\}$).   Since $n \geq 7$ we have that $\binom{n}{3}> 2\binom{n-1}{2}$ and so, without loss of generality,  we may assume that $\omega$ has at least $\binom{n-1}{2}+1$ entries equal to $e_1$.    From Lemma \ref{lemma1} we know that $\hat{\omega}=0\in \Lambda^{S^3}_{V_2}[n]$ is trivial. 
\end{proof}

\begin{corollary} \label{cor2} Let  $\omega=\otimes_{1\leq i<j<k\leq 6} (v_{i,j,k})\in \mathcal{G}_{\mathcal{B}_2}^{S^3}[6]\subseteq V_2^{\otimes 20}$, and  assume that there are at least $11$ entries equal to $e_1$ amongst the $v_{i,j,k}$,   then $\hat{\omega}=0\in \Lambda^{S^3}_{V_2}[6]$. 
\end{corollary}
\begin{proof}
It follows directly from Lemma \ref{lemma1} for $n=6$. 
\end{proof}

\begin{remark}
Notice that if $\omega\in \mathcal{G}_{\mathcal{B}_2}^{S^3}[6]$ such that  $\hat{\omega}\neq 0\in \Lambda^{S^3}_{V_2}[6]$,  then by Corollary \ref{cor2} the corresponding partition $\mathcal{P}_{\omega}=(\mathcal{H}_1,\mathcal{H}_2)$ of $K_6^3$ must be homogeneous (i.e. $\omega$ has ten entries equal to $e_1$ and ten entries equal to $e_2$). %We believe that a similar result should be true for any $d$, more precisely if $\mathcal{P}\in \mathcal{P}_{d}(K_{3d}^3)$ is not homogeneous then we expect that $\widehat{\omega_{\mathcal{P}}}=0\in \Lambda^{S^3}_{V_d}[3d]$, but we do not have a proof for this yet.
\end{remark}

\begin{lemma} (i) There is an action of the symmetric group $S_{n}$ on $\Lambda^{S^3}_{V_d}[n]$ given by 
$$\sigma\cdot \wedge^{S^3}_{1\leq i<j<k\leq n}(v_{i,j,k})=\wedge^{S^3}_{1\leq i<j<k\leq n}(v_{\sigma^{-1}(i),\sigma^{-1}(j),\sigma^{-1}(k)}).$$\\
(ii) There is an action of the group $GL_d(k)$ on $\Lambda^{S^3}_{V_d}[n]$ given by 
$$T\cdot \wedge^{S^3}_{1\leq i<j<k\leq n}(v_{i,j,k})=\wedge^{S^3}_{1\leq i<j<k\leq n}(T(v_{i,j,k})).$$
(iii) If $\mathcal{P}$ is a homogeneous $2$-partition of $K_6^3$ and $T\in GL_2(k)$ then 
$$T\cdot \widehat{\omega_{\mathcal{P}}}=det(T)^{10}\widehat{\omega_{\mathcal{P}}}.$$
In particular there exist  an action of the group $S_6\times S_2$ on $\Lambda^{S^3}_{V_2}[6]$, where the $S_6$ action is the one described in 1), and $S_2$ is the subgroup on $GL_2(k)$ generated by $\tau:V_2\to V_2$, $\tau(e_1)=e_2$ and $\tau(e_2)=e_1$. \label{lemma45}
\end{lemma}
\begin{proof} The first two statements follow directly from the definition of $\Lambda^{S^3}_{V_d}[n]$. 
For  statement $(iii)$ first notice that it is enough to check it for diagonal and elementary transformations. Take $T_1, T_2:V_2\to V_2$ 
\[T_1(e_s)=\begin{cases}
		\lambda e_1 \ \ \ \ s=1\\
		e_2 \ \ \ \ \ \ s=2,
\end{cases}\]
\[T_2(e_s)=\begin{cases}
		e_1 \ \ \ \ \ \ \ \ \ \ \ s=i \\
		e_2+\lambda e_1 \ \ \ \ s=2,
	\end{cases}\]
and consider $\mathcal{P}=(\mathcal{H}_1,\mathcal{H}_2)$ a homogeneous partition of  $K_6^3$. 
Then $$T_1\cdot \widehat{\omega_{\mathcal{P}}}=\lambda^{10}\widehat{\omega_{\mathcal{P}}}=det(T_1)^{10}\widehat{\omega_{\mathcal{P}}}$$ where the first equality is true because the tensor product is linear in each component, and because there are exactly ten  entries equal to $e_1$ in $\omega_{\mathcal{P}}$. 

Next we have $$T_2\cdot  \widehat{\omega_{\mathcal{P}}}=\sum_{j=1}^{2^{10}} \lambda^{k_j} \widehat{\omega_{\mathcal{P}^j}},$$ where the sum is taken over all $2$-partitions $\mathcal{P}^j=(\mathcal{H}^j_1,\mathcal{H}^j_2)$ of $K_6^3$ with the property that  
$E(\mathcal{H}_1)\subseteq E(\mathcal{H}^j_1)$ and $k_j=\vert E(\mathcal{H}^j_1)\vert -10$. Obviously the only homogeneous  partition among the $\mathcal{P}^j$'s is our initial partition $\mathcal{P}$, and so because of Corollary \ref{cor2} we get 
$$T_2\cdot \widehat{\omega_{\mathcal{H}}}=\widehat{\omega_{\mathcal{H}}}=det(T_2)^{10}\widehat{\omega_{\mathcal{H}}},$$
which completes our proof. 
\end{proof}

%While we do have an explicit basis for $\mathcal{T}^{S^3}_{V_d}[n]$, finding a basis for $\Lambda^{S^3}_{V_d}[n]$ is a much more difficult task.  In the rest of this section we will concentrate on computing  $dim_k(\Lambda^{S^3}_{V_2}[n])$ for all $n\geq 0$. In particular we will show that $dim_k(\Lambda^{S^3}_{V_2}[6])=1$ which will allow us to define the $det^{S^3}$ map. 

As mentioned above, if $\omega\in\mathcal{G}_{\mathcal{B}_2}^{S^3}[6]$ whose image in $\Lambda^{S^3}_{V_2}[6]$ is non-zero then its corresponding partition $\mathcal{P}_{\omega}$ must be homogeneous. There are exactly ${20 \choose 10}=\num{184756}$ homogeneous $2$-partitions of $K_6^3$, but not all of them give nonzero elements in $\Lambda^{S^3}_{V_2}[6]$.

For example, any partition $\mathcal{P}=(\mathcal{H}_1,\mathcal{H}_2)$ for which we can find $1\leq x<y<z<t\leq 6$ such that $\{x,y,z\}$, $\{x,y,t\}$, $\{x,z,t\}$ and $\{y,z,t\}\in E(\mathcal{H}_1)$ has the property that $ \widehat{\omega_{\mathcal{P}}}=0\in \Lambda^{S^3}_{V_2}[6]$. More generally we have the following. 

\begin{lemma}
Let $\mathcal{P}=(\mathcal{H}_1,\mathcal{H}_2)$ be a homogeneous $2$-partition of $K_6^3$. If $\mathcal{P}$ satisfies one of the  following conditions then $\widehat{\omega_{\mathcal{P}}}=0\in \Lambda^{S^3}_{V_2}[6]$. \\
(i)  There exist four distinct integers $1\leq x,y,z,t\leq 6$  and $1\leq i\leq 2$ such that $\{x,y,z\}$, $\{x,y,t\}$, $\{x,z,t\}$ and $\{y,z,t\}\in E(\mathcal{H}_{i})$.\\
(ii) There exist five distinct integers  $1\leq x,y,z,t,u\leq 6$  and $1\leq i\leq 2$ such that $\{x,y,z\}$, $\{x,y,t\}$, $\{x,z,t\}$, $\{y,z,u\}$, $\{y,t,u\}$ and $\{z,t,u\}\in E(\mathcal{H}_i)$.\\
%(iii) there exists five distinct integers $1\leq x,y,z,t,u\leq6$ and $1\leq i\leq2$ such that $\x,z,t)$, $(y,z,t)$, $(y,z,u)$, $(x,z,u)$, $(x,y,z)$, $(y,t,u)$, and $(x,t,u)\in\mathcal{H}_i$.\\ %this case is a repeat of case (ii), but was included and looks different when drawn
(iii) There exist six distinct integers $1\leq x,y,z,t,u,v\leq 6$  and $1\leq i\leq 2$ such that $\{x,y,u\}$, $\{x,y,v\}$, 
$\{y,z,u\}$, $\{x,z,u\}$, $\{x,t,v\}$,  $\{y,t,v\}$,  $\{y,z,t\}$, and $\{x,z,t\}\in E(\mathcal{H}_i)$.\\
(iv) There exist six distinct integers $1\leq x,y,z,t,u,v\leq 6$  and $1\leq i\leq 2$ such that $\{x,y,z\}$, $\{x,y,t\}$, $\{x,z,t\}$, $\{x,y,u\}$, $\{y,t,u\}$, $\{x,y,v\}$, $\{x,u,v\}$, $\{y,z,v\}$, $\{z,t,v\}$  and $\{t,u,v\}\in E(\mathcal{H}_i)$. \label{trivpar}
\end{lemma}

\begin{proof} We only give details for $(ii)$, the other statements are similar. Using relation \ref{rel1} for $(x,y,z,t)$ we get that 
$$\widehat{\omega_{\mathcal{P}}}=-\widehat{\omega_{\mathcal{P}^{(1)}}}-\widehat{\omega_{\mathcal{P}^{(2)}}}-\widehat{\omega_{\mathcal{P}^{(3)}}},$$ 
where $\mathcal{P}^{(j)}=(\mathcal{H}_1^{(j)},\mathcal{H}_2^{(j)})$ are distinct homogeneous $2$-partitions of $K_6^3$ such that $\{y,z,t\}\in E(\mathcal{H}_i^{(j)})$, and $\mathcal{P}^{(j)}$ coincide with $\mathcal{P}$ everywhere except maybe on the hyperedges $\{x,y,z\}$, $\{x,y,t\}$, $\{x,z,t\}$ and $\{y,z,t\}$. 

Notice that $\{y,z,u\}$, $\{y,t,u\}$, $\{z,t,u\}$ and $\{y,z,t\}\in E(\mathcal{P}_i^{(j)})$ for $1\leq j\leq 3$, and so by $(i)$ we get that $\widehat{\omega_{\mathcal{P}^{(j)}}}=0$ for $1\leq j\leq 3$ which proves our statement. 
\end{proof}

\begin{remark} 
 There are \num{184756} homogeneous $2$-partitions of $K_6^3$. 
One can use Lemma \ref{trivpar} and  MATLAB to sort out the trivial partitions. After this process we are still left with \num{13644} nontrivial homogeneous $2$-partitions of $K_6^3$, which we will denote by $\mathcal{P}_2^{h,nt}(K_{6}^3)$ (i.e. those partitions that are not listed in Lemma \ref{trivpar}). 

Recall that on  $\mathcal{P}_2^{h}(K_{6}^3)$ there is a natural action of the group $S_6\times S_2$. It is obvious that the action of $S_6\times S_2$ restricts to $\mathcal{P}_2^{h,nt}(K_{6}^3)$. Using MATLAB one can give a classification of the elements in $\mathcal{P}_2^{h,nt}(K_{6}^3)$ under this action and obtain $20$ equivalence classes. The details are presented in the Appendix.
\end{remark}

\begin{definition}
Let $\mathcal{P}=(\mathcal{H}_1,\mathcal{H}_2)\in  \mathcal{P}_2^{h}(K_{6}^3)$ and $1\leq x<y<z<t\leq 6$. We denote by $Pair(\mathcal{P},(x,y,z,t))$ the set of all homogeneous $2$-partitions of $K_6^3$ that coincide with $\mathcal{P}$  except maybe on the hyperedges $\{x,y,z\}$, $\{x,y,t\}$, $\{x,z,t\}$ and $\{y,z,t\}$.
\end{definition}

\begin{remark} We have one of the following three cases.\\
(i)  $Pair(\mathcal{P},(x,y,z,t))$ has one single element only if all the hyperedges $\{x,y,z\}$, $\{x,y,t\}$, $\{x,z,t\}$ and $\{y,z,t\}$ belong to $\mathcal{H}_1$, or all of them belong to $\mathcal{H}_2$. \\
(ii) If three of the hyperedges $\{x,y,z\}$, $\{x,y,t\}$, $\{x,z,t\}$ and $\{y,z,t\}$ belong to $\mathcal{H}_1$,  and one to  $\mathcal{H}_2$ (or the other way around) then $Pair(\mathcal{P},(x,y,z,t))$ has four elements. \\
(iii) If  two of the hyperedges $\{x,y,z\}$, $\{x,y,t\}$, $\{x,z,t\}$ and $\{y,z,t\}$ belong to $\mathcal{H}_1$  and two belong  to  $\mathcal{H}_2$ then $Pair(\mathcal{P},(x,y,z,t))$ has six elements.
\label{rempair}
\end{remark}
With these notations we have have following result. 
\begin{lemma} \label{act}
There exits a unique map $\varepsilon^{S^3}:\mathcal{P}_2^{h}(K_{6}^3)\to \{-4,-1,0,1\}$ such that\\
(i) if $\mathcal{P}\in  \mathcal{P}_2^{h}(K_{6}^3)$ and $1\leq x<y<z<t\leq 6$ then 
$$\sum_{\mathcal{K}\in Pair(\mathcal{P},(x,y,z,t))}\varepsilon^{S^3}(\mathcal{K})=0,$$
(ii) $\varepsilon^{S^3}$ takes value $1$ on the partition $\mathcal{P}^{(1)}$ from Figure \ref{figex1}.
\end{lemma}

\begin{proof}
Using relations \ref{rel1}, \ref{rel2}, \ref{rel3} and Remark \ref{rempair} one can write a system of linear equation that gives all the solutions for condition (i).  Using MATLAB one can show that the corresponding matrix has co-rank equal to $1$, and so because of condition (ii) we get a unique solution. 
\end{proof}

\begin{remark}
It it rather interesting  to notice that $\varepsilon^{S^3}(\mathcal{P})\neq 0$ if and only if $\mathcal{P}\in \mathcal{P}_2^{h,nt}(K_{6}^3)$. This is somehow similar with the results from \cite{edge} where $\varepsilon^{S^2}(\Gamma)\neq 0$ if and only if $\Gamma$ was cycle free. One can see that the partitions listed in  Lemma \ref{trivpar}  have a copy of $S^2$ (made of hyperedges/faces) either in $\mathcal{H}_1$ or in $\mathcal{H}_2$.  However, unlike $\varepsilon^{S^2}$, the map $\varepsilon^{S^3}$ takes also the value $-4$ which is a rather unexpected fact.  A table of values of $\varepsilon^{S^3}$  on all elements in $\mathcal{P}_2^{h,nt}(K_{6}^3)$ is given in the Appendix. \label{remsph}
\end{remark}

The  map $\varepsilon^{S^3}$ plays a role similar with the signature of a permutation and with the $\varepsilon^{S^2}$ map from \cite{edge}. It allows us to define a determinant-like function $det^{S^3}:V_2^{\otimes 20}\to k$. 

More precisely, take $v_{i,j,k}=\alpha_{i,j,k}e_1+\beta_{i,j,k}e_2\in V_2$ for $1\leq i<j<k\leq6$. For a $2$-partition $(\mathcal{H}_1,\mathcal{H}_2)\in\mathcal{P}_2^{h,nt}(K_{6}^3)$, define
\[M^{S^3}_{(\mathcal{H}_1,\mathcal{H}_2)}((v_{i,j,k})_{1\leq i<j<k\leq 6})=\prod_{\{u_1,v_1,w_1\}\in E(\mathcal{H}_1)}\alpha_{u_1,v_1,w_1}\prod_{\{u_2,v_2,w_2\}\in E(\mathcal{H}_2)}\beta_{u_2,v_2,w_2}.\]
Next, take $$Det^{S^3}:V_2^{20}\to k$$ determined by
\[Det^{S^3}((v_{i,j,k})_{1\leq i<j<k\leq6})=\sum_{(\mathcal{H}_1,\mathcal{H}_2)\in \mathcal{P}_2^{h,nt}(K_{6}^3)}\varepsilon^{S^3}(\mathcal{H}_1,\mathcal{H}_2)M^{S^3}_{(\mathcal{H}_1,\mathcal{H}_2)}((v_{i,j,k})_{1\leq i<j<k\leq6})\]
Notice that $Det^{S^3}$ is multi-linear and so we get a linear map $det^{S^3}:V_2^{\otimes 20}\to k$. The following theorem is the main result of this paper.

\begin{theorem} \label{thdet} The map $det^{S^3}:V_2^{\otimes 20}\to k$  is the unique (up to a constant) nontrivial linear map on $V_2^{\otimes 20}$ with the property that that $det^{S^3}(\otimes_{1\leq i<j<k\leq 6} (v_{i,j,k}))=0$ if there exist $1\leq x<y<z<t\leq 6$ such that $v_{x,y,z}=v_{x,y,t}=v_{x,z,t}=v_{y,z,t}$. 
In particular we have a unique  (up to a constant) nontrivial linear map  $det^{S^3}:\Lambda^{S^3}_{V_2}[6]\to k$. 
\end{theorem}
\begin{proof} This follows directly from Proposition \ref{proprel} and Lemma \ref{act}. 
\end{proof}
\begin{remark} An alternative proof (which does not use MATLAB) for the existence of the $det^{S^3}$ map is given in the Appendix. That approach is based on the fact that $det^{S^3}$ is invariant under the action of $SL_2(k)$. However that approach  does not prove the uniqueness. 
\end{remark}

\begin{corollary} \label{dim1}
 If $dim_k(V_2)=2$, then $dim_k\left(\Lambda^{S^3}_{V_2}[6]\right)=1$.
\end{corollary}
\begin{proof}
 It follows from Theorem \ref{thdet} and the definition of $\Lambda^{S^3}_{V_2}[6]$. 
\end{proof}

To conclude this sections we have the  following complete list  $dim_k(\Lambda_{V_2}^{S^3}[n])$. 
\begin{proposition} Let $V_2$ be a vector space  of dimension two. Then we have:\\
(i) $dim_k(\Lambda_{V_2}^{S^3}[0])=dim_k(\Lambda_{V_2}^{S^3}[1])=dim_k(\Lambda_{V_2}^{S^3}[2])=1$,\\
(ii) $dim_k(\Lambda_{V_2}^{S^3}[3])=2$,\\
(iii) $dim_k(\Lambda_{V_2}^{S^3}[4])=11$,\\
(iv) $dim_k(\Lambda_{V_2}^{S^3}[5])=62$,\\
(v) $dim_k(\Lambda_{V_2}^{S^3}[6])=1$,\\
(vi) $dim_k(\Lambda_{V_2}^{S^3}[n])=0$ if $n\geq 7$.
\end{proposition}
\begin{proof}
Most of the results are either trivial or were already covered in this section. The only interesting cases are $n=4$ and  $n=5$. 
 
Recall that  $\mathcal{G}_{\mathcal{B}_2}^{S^3}[n]$ is in bijection with $\mathcal{P}_2(K_n^3)$. We denote by 
$\mathcal{P}_2(K_n^3)^{p,q}$ those $2$-partitions $(\mathcal{H}_1,\mathcal{H}_2)$ of $K_n^3$ with the property that $\vert E(\mathcal{H}_1)\vert=p$ and  $\vert E(\mathcal{H}_2)\vert=q$  (obviously we must have that $p+q=\binom{n}{3}$). We denote by $\Lambda_{V_2}^{S^3}[n]^{p,q}$ the corresponding subspace in $\Lambda_{V_2}^{S^3}[n]$. 

Since the relations in $\Lambda_{V_2}^{S^3}[n]$ are homogeneous in the number of $e_1$'s and $e_2$'s, it follows that if $\omega_{p,q}\in \Lambda_{V_2}^{S^3}[n]^{p,q}$, and  $\sum \omega_{p,q}=0\in \Lambda_{V_2}^{S^3}[n]$ then we must have that $\omega_{p,q}=0$ for all $p$, $q$. 

Let's first consider the case $n=4$.  If $\mathcal{P}\in \mathcal{P}_2(K_4^3)^{4,0}$, or $\mathcal{P}\in \mathcal{P}_2(K_4^3)^{0,4}$ then obviously $\widehat{x_{\mathcal{P}}}=0\in \Lambda_{V_2}^{S^3}[4]$. 

There are four elements  $\mathcal{P}\in \mathcal{P}_2(K_4^3)^{3,1}$, and the only relation among them is listed in Figure \ref{figR2}. This means that that we get three linearly independent vectors that generate $\Lambda_{V_2}^{S^3}[4]^{3,1}$. The case $\Lambda_{V_2}^{S^3}[4]^{1,3}$ is similar. 

There are six elements $\mathcal{P}\in \mathcal{P}_2(K_4^3)^{2,2}$, and the only relation among them is listed in Figure \ref{figR3}. This means that that we get five linearly independent vectors that generate $\Lambda_{V_2}^{S^3}[4]^{2,2}$.
To conclude, we have 
$$dim_k(\Lambda_{V_2}^{S^3}[4])=dim_k(\Lambda_{V_2}^{S^3}[4]^{3,1})+dim_k(\Lambda_{V_2}^{S^3}[4]^{1,3})+dim_k(\Lambda_{V_2}^{S^3}[4]^{2,2})=3+3+5=11.$$ 

The case $n=5$  is similar but computationally heavier, so we had to use MATLAB. The interesting cases are $(p,q)\in\{(6,4),(5,5),(4,6)\}$. One can show that $dim_k(\Lambda_{V_2}^{S^3}[5]^{6,4})=15=dim_k(\Lambda_{V_2}^{S^3}[5]^{4,6})$ and $dim_k(\Lambda_{V_2}^{S^3}[5]^{5,5})=32$. This gives that $dim_k(\Lambda_{V_2}^{S^3}[5])=15+15+32=62$. 
\end{proof}

\section{Some Remarks}

In the previous section we computed $dim_k\left(\Lambda^{S^3}_{V_2}[n]\right)$ for all $n\geq 0$. It  would be interesting to understand how $\Lambda^{S^3}_{V_d}[n]$ behaves for any $d$. Based on the results from this paper we have the following question. 
\begin{question}
 Suppose $dim_k(V_d)=d$. Is it true that $dim_k\left(\Lambda^{S^3}_{V_d}[n]\right)=0$ for $n>3d$? Is it true that $dim_k\left(\Lambda^{S^3}_{V_d}[3d]\right)=1$?
\end{question}

Note that this question  is in the spirit of the results from \cite{sta2}.  The particular challenge here is that even for the simplest case $d=3$, the  computation is less feasible, as there are on the order of $1.17\times10^{38}$ possible homogeneous $3$-partitions for the hypergraph $K_9^3$, and our proof for Theorem \ref{dim1} is computational. A different, more theoretical approach is necessary in order to solve this problem.  

\begin{remark}\label{geometry}
Recall from \cite{edge}  that if $\Gamma=(\Gamma_1,\dots,\Gamma_d)$ is a  homogeneous $d$-partitions of $K_{2d}$ such that $\widehat{f_{\Gamma}}\neq 0\in\Lambda_{V_d}^{S^2}[2d]$ then $\Gamma$ must be cycle-free. 
It would be interesting to find a similar result for elements in $\mathcal{G}_{\mathcal{B}_d}^{S^3}[3d]$. 
More precisely, let $\mathcal{P}=(\mathcal{H}_1,\dots, \mathcal{H}_d)$ be a homogeneous $d$-partition of the hypergraph $K_{3d}^3$, find a combinatorial property of $\mathcal{P}$ such that $\widehat{\omega_{\mathcal{P}}}\neq 0$. 
\end{remark} 
\begin{remark}
Since $det^{S^3}$ is invariant under the action of $SL_2(k)$ we know from general theory of invariant functions (see \cite{stu}) that condition $det^{S^3}(\otimes_{1\leq i<j<k\leq 6}(v_{i,j,k})=0$ must have a geometrical interpretation. It would be interesting to find an explicit description similar with the results from \cite{sv} for the $det^{S^2}$ map. 
\end{remark}

One can try to generalize the construction from this paper to any sphere $S^r$ as follows. Take 
$$\mathcal{T}_{V}^{S^r}[n]=V^{\otimes{ {n}\choose{r}}}$$
and define   $\mathcal{E}^{S^r}_V[n]$ to be the subspace of 
$\mathcal{T}^{S^r}_V[n]$ generated by simple tensors 
$$\otimes_{1\leq i_1<i_2<\dots<i_r\leq n} (v_{i_1,i_2,...,i_{r}})$$ 
with the property that there exists $1\leq x_1<x_2<\dots<x_{r+1}\leq n$ such that $$v_{x_1,x_2,...,x_{r}}=v_{x_1,x_2,...,x_{r-1},x_{r+1}}=\dots=v_{x_1,x_3,...,x_{r+1}}=v_{x_2,x_3,...,x_{r+1}}.$$ 

\begin{definition} With the above notations we define
$$\Lambda^{S^r}_{V}[n]= \frac{\mathcal{T}^{S^r}_V[n]}{\mathcal{E}^{S^r}_V[n]},$$
and 
$$\Lambda^{S^r}_V=\bigoplus_{n\geq 0}\Lambda^{S^r}_V[n].$$
\end{definition}
%We have the following question
\begin{question}
 Suppose $dim(V_d)=d$. Is it true that $dim_k\left(\Lambda^{S^r}_{V_d}[n]\right)=0$ for $n>rd$? Is it true that $dim_k\left(\Lambda^{S^r}_{V_d}[rd]\right)=1$?
\end{question}

\begin{remark} As we recalled in introduction, $\Lambda_V$ has algebra structure on it, and $\Lambda^{S^2}_V$ is a GSC-operad.   It is natural to ask if there is more structure on $\Lambda^{S^3}_{V}$. As far as we can tell there is no obvious algebra, or GSC-operad structure on it. We expect that some operad-like structure exists on $\Lambda^{S^3}_{V}$, and we plan to investigate this problem in a follow-up paper. 
\end{remark}

%\begin{remark} One can think of the construction from this paper as some kind of linear Ramsey theory. The dictionary should be be obvious at this point. We already discussed about the bijection between $\mathcal{G}_{\mathcal{B}_d}^{S^3}[n]$ and $\mathcal{P}_d(K_n^3)$. It is not difficult to see that there is a similar bijection between a basis $\mathcal{G}_{\mathcal{B}_d}^{S^r}[n]$ of $\mathcal{T}_{V}^{S^r}[n]$ and the set $\mathcal{P}_d(K_n^r)$ of $d$-partitions of the $r$-uniform complete hypergraph $K_n^r$.   If we think that the basis $\mathcal{B}_d=\{e_1,\dots,e_d\}$ is the set of colors,  then  the elements in $\mathcal{G}_{\mathcal{B}_d}^{S^3r}[n]$ correspond to coloring of hypergraph $K_n^r$. The main difference from the usual theory is that we allow colorings with linear combinations of colors.  \end{remark}

\section*{Acknowledgment}
We thank Alin Stancu for some comments and discussions.

\appendix 

\maketitle

\section{An explicit formula for $det^{S^3}$}

It follows from Proposition \ref{lemma45} that the map $det^{S^3}$ is invariant under the action of $SL_2(k)$ on $V_2$. From the general results of invariant theory (see \cite{stu}), it follows that $det^{S^3}(\otimes_{1\leq i<j<k\leq 6} (v_{i,j,k}))$ can be written as sum of product of determinants of two by two matrices with columns consisting of the vectors $v_{i,j,k}$. In this section we give an explicit formula for $det^{S^3}$. This is also an alternative proof for the existence of the map $det^{S^3}$. 

For two vectors $u=(a,b)$ and $v=(c,d)$ we denote by $[u,v]$ the determinant of the matrix that has the first column  equal to $u$ and the second column equal to $v$
$$[u,v]=ad-bc.$$

\begin{landscape}

\begin{proposition} Let $v_{i,j,k}\in V_2$ for all $1\leq i<j<k\leq 6$, the we have

\begin{eqnarray*}
&det^{S^3}(\otimes_{1\leq i<j<k\leq 6} (v_{i,j,k}))=&\\
& [v_{1,2,3},v_{2,3,4}] [v_{1,2,4},v_{2,4,5}] [v_{1,2,5},v_{2,5,3}][v_{1,2,6},v_{2,6,5}][v_{1,3,4},v_{3,4,6}][v_{1,3,5},v_{3,5,4}][v_{1,3,6},v_{3,6,2}][v_{1,4,5},v_{4,5,6}][v_{1,4,6},v_{4,6,2}][v_{1,5,6},v_{5,6,3}]+&\\
& [v_{1,2,3},v_{2,3,5}] [v_{1,2,4},v_{2,4,3}] [v_{1,2,5},v_{2,5,4}][v_{1,2,6},v_{2,6,5}][v_{1,3,4},v_{3,4,6}][v_{1,3,5},v_{3,5,6}][v_{1,3,6},v_{3,6,2}][v_{1,4,5},v_{4,5,3}][v_{1,4,6},v_{4,6,2}][v_{1,5,6},v_{5,6,4}]+&\\
& [v_{1,2,3},v_{2,3,5}] [v_{1,2,4},v_{2,4,3}] [v_{1,2,5},v_{2,5,4}][v_{1,2,6},v_{2,6,4}][v_{1,3,4},v_{3,4,5}][v_{1,3,5},v_{3,5,6}][v_{1,3,6},v_{3,6,2}][v_{1,4,5},v_{4,5,6}][v_{1,4,6},v_{4,6,3}][v_{1,5,6},v_{5,6,2}]+&\\
& [v_{1,2,3},v_{2,3,4}] [v_{1,2,4},v_{2,4,5}] [v_{1,2,5},v_{2,5,3}][v_{1,2,6},v_{2,6,4}][v_{1,3,4},v_{3,4,6}][v_{1,3,5},v_{3,5,6}][v_{1,3,6},v_{3,6,2}][v_{1,4,5},v_{4,5,3}][v_{1,4,6},v_{4,6,5}][v_{1,5,6},v_{5,6,2}]+&\\
& [v_{1,2,3},v_{2,3,4}] [v_{1,2,4},v_{2,4,5}] [v_{1,2,5},v_{2,5,3}][v_{1,2,6},v_{2,6,3}][v_{1,3,4},v_{3,4,5}][v_{1,3,5},v_{3,5,6}][v_{1,3,6},v_{3,6,4}][v_{1,4,5},v_{4,5,6}][v_{1,4,6},v_{4,6,2}][v_{1,5,6},v_{5,6,2}]+&\\
& [v_{1,2,3},v_{2,3,5}] [v_{1,2,4},v_{2,4,3}] [v_{1,2,5},v_{2,5,4}][v_{1,2,6},v_{2,6,3}][v_{1,3,4},v_{3,4,6}][v_{1,3,5},v_{3,5,4}][v_{1,3,6},v_{3,6,5}][v_{1,4,5},v_{4,5,6}][v_{1,4,6},v_{4,6,2}][v_{1,5,6},v_{5,6,2}]+&\\
& [v_{1,2,3},v_{2,3,4}] [v_{1,2,4},v_{2,4,6}] [v_{1,2,5},v_{2,5,6}][v_{1,2,6},v_{2,6,3}][v_{1,3,4},v_{3,4,5}][v_{1,3,5},v_{3,5,2}][v_{1,3,6},v_{3,6,4}][v_{1,4,5},v_{4,5,2}][v_{1,4,6},v_{4,6,5}][v_{1,5,6},v_{5,6,3}]+&\\
& [v_{1,2,3},v_{2,3,6}] [v_{1,2,4},v_{2,4,3}] [v_{1,2,5},v_{2,5,6}][v_{1,2,6},v_{2,6,4}][v_{1,3,4},v_{3,4,5}][v_{1,3,5},v_{3,5,2}][v_{1,3,6},v_{3,6,5}][v_{1,4,5},v_{4,5,2}][v_{1,4,6},v_{4,6,3}][v_{1,5,6},v_{5,6,4}]+&\\
& [v_{1,2,3},v_{2,3,6}] [v_{1,2,4},v_{2,4,3}] [v_{1,2,5},v_{2,5,4}][v_{1,2,6},v_{2,6,4}][v_{1,3,4},v_{3,4,6}][v_{1,3,5},v_{3,5,2}][v_{1,3,6},v_{3,6,5}][v_{1,4,5},v_{4,5,3}][v_{1,4,6},v_{4,6,5}][v_{1,5,6},v_{5,6,2}]+&\\
& [v_{1,2,3},v_{2,3,4}] [v_{1,2,4},v_{2,4,6}] [v_{1,2,5},v_{2,5,4}][v_{1,2,6},v_{2,6,3}][v_{1,3,4},v_{3,4,5}][v_{1,3,5},v_{3,5,2}][v_{1,3,6},v_{3,6,5}][v_{1,4,5},v_{4,5,6}][v_{1,4,6},v_{4,6,3}][v_{1,5,6},v_{5,6,2}]+&\\
& [v_{1,2,3},v_{2,3,4}] [v_{1,2,4},v_{2,4,6}] [v_{1,2,5},v_{2,5,3}][v_{1,2,6},v_{2,6,3}][v_{1,3,4},v_{3,4,6}][v_{1,3,5},v_{3,5,4}][v_{1,3,6},v_{3,6,5}][v_{1,4,5},v_{4,5,2}][v_{1,4,6},v_{4,6,5}][v_{1,5,6},v_{5,6,2}]+&\\
& [v_{1,2,3},v_{2,3,6}] [v_{1,2,4},v_{2,4,3}] [v_{1,2,5},v_{2,5,3}][v_{1,2,6},v_{2,6,4}][v_{1,3,4},v_{3,4,5}][v_{1,3,5},v_{3,5,6}][v_{1,3,6},v_{3,6,4}][v_{1,4,5},v_{4,5,2}][v_{1,4,6},v_{4,6,5}][v_{1,5,6},v_{5,6,2}]+&\\
& [v_{1,2,3},v_{2,3,5}] [v_{1,2,4},v_{2,4,6}] [v_{1,2,5},v_{2,5,6}][v_{1,2,6},v_{2,6,3}][v_{1,3,4},v_{3,4,2}][v_{1,3,5},v_{3,5,4}][v_{1,3,6},v_{3,6,5}][v_{1,4,5},v_{4,5,2}][v_{1,4,6},v_{4,6,3}][v_{1,5,6},v_{5,6,4}]+&\\
& [v_{1,2,3},v_{2,3,6}] [v_{1,2,4},v_{2,4,6}] [v_{1,2,5},v_{2,5,3}][v_{1,2,6},v_{2,6,5}][v_{1,3,4},v_{3,4,2}][v_{1,3,5},v_{3,5,4}][v_{1,3,6},v_{3,6,4}][v_{1,4,5},v_{4,5,2}][v_{1,4,6},v_{4,6,5}][v_{1,5,6},v_{5,6,3}]+&\\
& [v_{1,2,3},v_{2,3,6}] [v_{1,2,4},v_{2,4,5}] [v_{1,2,5},v_{2,5,3}][v_{1,2,6},v_{2,6,5}][v_{1,3,4},v_{3,4,2}][v_{1,3,5},v_{3,5,6}][v_{1,3,6},v_{3,6,4}][v_{1,4,5},v_{4,5,3}][v_{1,4,6},v_{4,6,2}][v_{1,5,6},v_{5,6,4}]+&\\
& [v_{1,2,3},v_{2,3,5}] [v_{1,2,4},v_{2,4,5}] [v_{1,2,5},v_{2,5,6}][v_{1,2,6},v_{2,6,3}][v_{1,3,4},v_{3,4,2}][v_{1,3,5},v_{3,5,4}][v_{1,3,6},v_{3,6,4}][v_{1,4,5},v_{4,5,6}][v_{1,4,6},v_{4,6,2}][v_{1,5,6},v_{5,6,3}]+&\\
& [v_{1,2,3},v_{2,3,5}] [v_{1,2,4},v_{2,4,3}] [v_{1,2,5},v_{2,5,6}][v_{1,2,6},v_{2,6,3}][v_{1,3,4},v_{3,4,5}][v_{1,3,5},v_{3,5,6}][v_{1,3,6},v_{3,6,4}][v_{1,4,5},v_{4,5,2}][v_{1,4,6},v_{4,6,2}][v_{1,5,6},v_{5,6,4}]+&\\
& [v_{1,2,3},v_{2,3,6}] [v_{1,2,4},v_{2,4,3}] [v_{1,2,5},v_{2,5,3}][v_{1,2,6},v_{2,6,5}][v_{1,3,4},v_{3,4,6}][v_{1,3,5},v_{3,5,4}][v_{1,3,6},v_{3,6,5}][v_{1,4,5},v_{4,5,2}][v_{1,4,6},v_{4,6,2}][v_{1,5,6},v_{5,6,4}]+&\\
& [v_{1,2,3},v_{2,3,6}] [v_{1,2,4},v_{2,4,5}] [v_{1,2,5},v_{2,5,6}][v_{1,2,6},v_{2,6,4}][v_{1,3,4},v_{3,4,2}][v_{1,3,5},v_{3,5,2}][v_{1,3,6},v_{3,6,4}][v_{1,4,5},v_{4,5,3}][v_{1,4,6},v_{4,6,5}][v_{1,5,6},v_{5,6,3}]+&\\
& [v_{1,2,3},v_{2,3,6}] [v_{1,2,4},v_{2,4,6}] [v_{1,2,5},v_{2,5,4}][v_{1,2,6},v_{2,6,5}][v_{1,3,4},v_{3,4,2}][v_{1,3,5},v_{3,5,2}][v_{1,3,6},v_{3,6,5}][v_{1,4,5},v_{4,5,3}][v_{1,4,6},v_{4,6,3}][v_{1,5,6},v_{5,6,4}]+&\\
& [v_{1,2,3},v_{2,3,5}] [v_{1,2,4},v_{2,4,6}] [v_{1,2,5},v_{2,5,4}][v_{1,2,6},v_{2,6,5}][v_{1,3,4},v_{3,4,2}][v_{1,3,5},v_{3,5,4}][v_{1,3,6},v_{3,6,2}][v_{1,4,5},v_{4,5,6}][v_{1,4,6},v_{4,6,3}][v_{1,5,6},v_{5,6,3}]+&\\
& [v_{1,2,3},v_{2,3,5}] [v_{1,2,4},v_{2,4,5}] [v_{1,2,5},v_{2,5,6}][v_{1,2,6},v_{2,6,4}][v_{1,3,4},v_{3,4,2}][v_{1,3,5},v_{3,5,6}][v_{1,3,6},v_{3,6,2}][v_{1,4,5},v_{4,5,3}][v_{1,4,6},v_{4,6,3}][v_{1,5,6},v_{5,6,4}]+&\\
& [v_{1,2,3},v_{2,3,4}] [v_{1,2,4},v_{2,4,5}] [v_{1,2,5},v_{2,5,6}][v_{1,2,6},v_{2,6,4}][v_{1,3,4},v_{3,4,5}][v_{1,3,5},v_{3,5,2}][v_{1,3,6},v_{3,6,2}][v_{1,4,5},v_{4,5,6}][v_{1,4,6},v_{4,6,3}][v_{1,5,6},v_{5,6,3}]+&\\
& [v_{1,2,3},v_{2,3,4}] [v_{1,2,4},v_{2,4,6}] [v_{1,2,5},v_{2,5,4}][v_{1,2,6},v_{2,6,5}][v_{1,3,4},v_{3,4,6}][v_{1,3,5},v_{3,5,2}][v_{1,3,6},v_{3,6,2}][v_{1,4,5},v_{4,5,3}][v_{1,4,6},v_{4,6,5}][v_{1,5,6},v_{5,6,3}]-&\\
\end{eqnarray*}
\begin{eqnarray*}
& [v_{1,2,3},v_{2,3,4}] [v_{1,2,4},v_{2,4,6}] [v_{1,2,5},v_{2,5,4}][v_{1,2,6},v_{2,6,3}][v_{1,3,4},v_{3,4,6}][v_{1,3,5},v_{3,5,2}][v_{1,3,6},v_{3,6,5}][v_{1,4,5},v_{4,5,3}][v_{1,4,6},v_{4,6,5}][v_{1,5,6},v_{5,6,2}]-&\\
& [v_{1,2,3},v_{2,3,6}] [v_{1,2,4},v_{2,4,3}] [v_{1,2,5},v_{2,5,3}][v_{1,2,6},v_{2,6,4}][v_{1,3,4},v_{3,4,6}][v_{1,3,5},v_{3,5,4}][v_{1,3,6},v_{3,6,5}][v_{1,4,5},v_{4,5,2}][v_{1,4,6},v_{4,6,5}][v_{1,5,6},v_{5,6,2}]-&\\
& [v_{1,2,3},v_{2,3,5}] [v_{1,2,4},v_{2,4,5}] [v_{1,2,5},v_{2,5,6}][v_{1,2,6},v_{2,6,3}][v_{1,3,4},v_{3,4,2}][v_{1,3,5},v_{3,5,6}][v_{1,3,6},v_{3,6,4}][v_{1,4,5},v_{4,5,3}][v_{1,4,6},v_{4,6,2}][v_{1,5,6},v_{5,6,4}]-&\\
& [v_{1,2,3},v_{2,3,6}] [v_{1,2,4},v_{2,4,3}] [v_{1,2,5},v_{2,5,3}][v_{1,2,6},v_{2,6,5}][v_{1,3,4},v_{3,4,5}][v_{1,3,5},v_{3,5,6}][v_{1,3,6},v_{3,6,4}][v_{1,4,5},v_{4,5,2}][v_{1,4,6},v_{4,6,2}][v_{1,5,6},v_{5,6,4}]-&\\
& [v_{1,2,3},v_{2,3,5}] [v_{1,2,4},v_{2,4,5}] [v_{1,2,5},v_{2,5,6}][v_{1,2,6},v_{2,6,4}][v_{1,3,4},v_{3,4,2}][v_{1,3,5},v_{3,5,4}][v_{1,3,6},v_{3,6,2}][v_{1,4,5},v_{4,5,6}][v_{1,4,6},v_{4,6,3}][v_{1,5,6},v_{5,6,3}]-&\\
& [v_{1,2,3},v_{2,3,4}] [v_{1,2,4},v_{2,4,6}] [v_{1,2,5},v_{2,5,4}][v_{1,2,6},v_{2,6,5}][v_{1,3,4},v_{3,4,5}][v_{1,3,5},v_{3,5,2}][v_{1,3,6},v_{3,6,2}][v_{1,4,5},v_{4,5,6}][v_{1,4,6},v_{4,6,3}][v_{1,5,6},v_{5,6,3}]-&\\
& [v_{1,2,3},v_{2,3,4}] [v_{1,2,4},v_{2,4,5}] [v_{1,2,5},v_{2,5,3}][v_{1,2,6},v_{2,6,4}][v_{1,3,4},v_{3,4,5}][v_{1,3,5},v_{3,5,6}][v_{1,3,6},v_{3,6,2}][v_{1,4,5},v_{4,5,6}][v_{1,4,6},v_{4,6,3}][v_{1,5,6},v_{5,6,2}]-&\\
& [v_{1,2,3},v_{2,3,5}] [v_{1,2,4},v_{2,4,3}] [v_{1,2,5},v_{2,5,4}][v_{1,2,6},v_{2,6,3}][v_{1,3,4},v_{3,4,5}][v_{1,3,5},v_{3,5,6}][v_{1,3,6},v_{3,6,4}][v_{1,4,5},v_{4,5,6}][v_{1,4,6},v_{4,6,2}][v_{1,5,6},v_{5,6,2}]-&\\
& [v_{1,2,3},v_{2,3,6}] [v_{1,2,4},v_{2,4,6}] [v_{1,2,5},v_{2,5,3}][v_{1,2,6},v_{2,6,5}][v_{1,3,4},v_{3,4,2}][v_{1,3,5},v_{3,5,4}][v_{1,3,6},v_{3,6,5}][v_{1,4,5},v_{4,5,2}][v_{1,4,6},v_{4,6,3}][v_{1,5,6},v_{5,6,4}]-&\\
& [v_{1,2,3},v_{2,3,5}] [v_{1,2,4},v_{2,4,3}] [v_{1,2,5},v_{2,5,6}][v_{1,2,6},v_{2,6,3}][v_{1,3,4},v_{3,4,6}][v_{1,3,5},v_{3,5,4}][v_{1,3,6},v_{3,6,5}][v_{1,4,5},v_{4,5,2}][v_{1,4,6},v_{4,6,2}][v_{1,5,6},v_{5,6,4}]-&\\
& [v_{1,2,3},v_{2,3,6}] [v_{1,2,4},v_{2,4,6}] [v_{1,2,5},v_{2,5,4}][v_{1,2,6},v_{2,6,5}][v_{1,3,4},v_{3,4,2}][v_{1,3,5},v_{3,5,2}][v_{1,3,6},v_{3,6,4}][v_{1,4,5},v_{4,5,3}][v_{1,4,6},v_{4,6,5}][v_{1,5,6},v_{5,6,3}]-&\\
& [v_{1,2,3},v_{2,3,4}] [v_{1,2,4},v_{2,4,5}] [v_{1,2,5},v_{2,5,6}][v_{1,2,6},v_{2,6,4}][v_{1,3,4},v_{3,4,6}][v_{1,3,5},v_{3,5,2}][v_{1,3,6},v_{3,6,2}][v_{1,4,5},v_{4,5,3}][v_{1,4,6},v_{4,6,5}][v_{1,5,6},v_{5,6,3}]-&\\
& [v_{1,2,3},v_{2,3,5}] [v_{1,2,4},v_{2,4,3}] [v_{1,2,5},v_{2,5,4}][v_{1,2,6},v_{2,6,5}][v_{1,3,4},v_{3,4,6}][v_{1,3,5},v_{3,5,4}][v_{1,3,6},v_{3,6,2}][v_{1,4,5},v_{4,5,6}][v_{1,4,6},v_{4,6,2}][v_{1,5,6},v_{5,6,3}]-&\\
& [v_{1,2,3},v_{2,3,4}] [v_{1,2,4},v_{2,4,5}] [v_{1,2,5},v_{2,5,3}][v_{1,2,6},v_{2,6,3}][v_{1,3,4},v_{3,4,6}][v_{1,3,5},v_{3,5,4}][v_{1,3,6},v_{3,6,5}][v_{1,4,5},v_{4,5,6}][v_{1,4,6},v_{4,6,2}][v_{1,5,6},v_{5,6,2}]-&\\
& [v_{1,2,3},v_{2,3,6}] [v_{1,2,4},v_{2,4,3}] [v_{1,2,5},v_{2,5,6}][v_{1,2,6},v_{2,6,4}][v_{1,3,4},v_{3,4,5}][v_{1,3,5},v_{3,5,2}][v_{1,3,6},v_{3,6,4}][v_{1,4,5},v_{4,5,2}][v_{1,4,6},v_{4,6,5}][v_{1,5,6},v_{5,6,3}]-&\\
& [v_{1,2,3},v_{2,3,4}] [v_{1,2,4},v_{2,4,6}] [v_{1,2,5},v_{2,5,3}][v_{1,2,6},v_{2,6,3}][v_{1,3,4},v_{3,4,5}][v_{1,3,5},v_{3,5,6}][v_{1,3,6},v_{3,6,4}][v_{1,4,5},v_{4,5,2}][v_{1,4,6},v_{4,6,5}][v_{1,5,6},v_{5,6,2}]-&\\
& [v_{1,2,3},v_{2,3,6}] [v_{1,2,4},v_{2,4,5}] [v_{1,2,5},v_{2,5,6}][v_{1,2,6},v_{2,6,4}][v_{1,3,4},v_{3,4,2}][v_{1,3,5},v_{3,5,2}][v_{1,3,6},v_{3,6,5}][v_{1,4,5},v_{4,5,3}][v_{1,4,6},v_{4,6,3}][v_{1,5,6},v_{5,6,4}]-&\\
& [v_{1,2,3},v_{2,3,5}] [v_{1,2,4},v_{2,4,6}] [v_{1,2,5},v_{2,5,4}][v_{1,2,6},v_{2,6,5}][v_{1,3,4},v_{3,4,2}][v_{1,3,5},v_{3,5,6}][v_{1,3,6},v_{3,6,2}][v_{1,4,5},v_{4,5,3}][v_{1,4,6},v_{4,6,3}][v_{1,5,6},v_{5,6,4}]-&\\
& [v_{1,2,3},v_{2,3,4}] [v_{1,2,4},v_{2,4,5}] [v_{1,2,5},v_{2,5,3}][v_{1,2,6},v_{2,6,5}][v_{1,3,4},v_{3,4,6}][v_{1,3,5},v_{3,5,6}][v_{1,3,6},v_{3,6,2}][v_{1,4,5},v_{4,5,3}][v_{1,4,6},v_{4,6,2}][v_{1,5,6},v_{5,6,4}]-&\\
& [v_{1,2,3},v_{2,3,5}] [v_{1,2,4},v_{2,4,3}] [v_{1,2,5},v_{2,5,4}][v_{1,2,6},v_{2,6,4}][v_{1,3,4},v_{3,4,6}][v_{1,3,5},v_{3,5,6}][v_{1,3,6},v_{3,6,2}][v_{1,4,5},v_{4,5,3}][v_{1,4,6},v_{4,6,5}][v_{1,5,6},v_{5,6,2}]-&\\
 & [v_{1,2,3},v_{2,3,4}] [v_{1,2,4},v_{2,4,6}] [v_{1,2,5},v_{2,5,6}][v_{1,2,6},v_{2,6,3}][v_{1,3,4},v_{3,4,5}][v_{1,3,5},v_{3,5,2}][v_{1,3,6},v_{3,6,5}][v_{1,4,5},v_{4,5,2}][v_{1,4,6},v_{4,6,3}][v_{1,5,6},v_{5,6,4}]-&\\
& [v_{1,2,3},v_{2,3,6}] [v_{1,2,4},v_{2,4,3}] [v_{1,2,5},v_{2,5,4}][v_{1,2,6},v_{2,6,4}][v_{1,3,4},v_{3,4,5}][v_{1,3,5},v_{3,5,2}][v_{1,3,6},v_{3,6,5}][v_{1,4,5},v_{4,5,6}][v_{1,4,6},v_{4,6,3}][v_{1,5,6},v_{5,6,2}]-&\\
& [v_{1,2,3},v_{2,3,5}] [v_{1,2,4},v_{2,4,6}] [v_{1,2,5},v_{2,5,6}][v_{1,2,6},v_{2,6,3}][v_{1,3,4},v_{3,4,2}][v_{1,3,5},v_{3,5,4}][v_{1,3,6},v_{3,6,4}][v_{1,4,5},v_{4,5,2}][v_{1,4,6},v_{4,6,5}][v_{1,5,6},v_{5,6,3}]-&\\
& [v_{1,2,3},v_{2,3,6}] [v_{1,2,4},v_{2,4,5}] [v_{1,2,5},v_{2,5,3}][v_{1,2,6},v_{2,6,5}][v_{1,3,4},v_{3,4,2}][v_{1,3,5},v_{3,5,4}][v_{1,3,6},v_{3,6,4}][v_{1,4,5},v_{4,5,6}][v_{1,4,6},v_{4,6,2}][v_{1,5,6},v_{5,6,3}]+&\\
\end{eqnarray*}
\begin{eqnarray*}
& [v_{1,2,3},v_{2,3,5}] [v_{1,2,4},v_{2,4,6}] [v_{1,2,5},v_{2,5,3}][v_{1,2,6},v_{2,6,3}][v_{1,3,4},v_{3,4,2}][v_{1,3,5},v_{3,5,6}][v_{1,3,6},v_{3,6,4}][v_{1,4,5},v_{4,5,3}][v_{1,4,6},v_{4,6,5}][v_{1,5,6},v_{5,6,2}]+&\\
& [v_{1,2,3},v_{2,3,6}] [v_{1,2,4},v_{2,4,5}] [v_{1,2,5},v_{2,5,3}][v_{1,2,6},v_{2,6,4}][v_{1,3,4},v_{3,4,2}][v_{1,3,5},v_{3,5,4}][v_{1,3,6},v_{3,6,5}][v_{1,4,5},v_{4,5,6}][v_{1,4,6},v_{4,6,3}][v_{1,5,6},v_{5,6,2}]+&\\
& [v_{1,2,3},v_{2,3,4}] [v_{1,2,4},v_{2,4,5}] [v_{1,2,5},v_{2,5,6}][v_{1,2,6},v_{2,6,3}][v_{1,3,4},v_{3,4,6}][v_{1,3,5},v_{3,5,2}][v_{1,3,6},v_{3,6,5}][v_{1,4,5},v_{4,5,3}][v_{1,4,6},v_{4,6,2}][v_{1,5,6},v_{5,6,4}]+&\\
& [v_{1,2,3},v_{2,3,6}] [v_{1,2,4},v_{2,4,3}] [v_{1,2,5},v_{2,5,4}][v_{1,2,6},v_{2,6,5}][v_{1,3,4},v_{3,4,5}][v_{1,3,5},v_{3,5,2}][v_{1,3,6},v_{3,6,4}][v_{1,4,5},v_{4,5,6}][v_{1,4,6},v_{4,6,2}][v_{1,5,6},v_{5,6,3}]+&\\
& [v_{1,2,3},v_{2,3,5}] [v_{1,2,4},v_{2,4,3}] [v_{1,2,5},v_{2,5,6}][v_{1,2,6},v_{2,6,4}][v_{1,3,4},v_{3,4,6}][v_{1,3,5},v_{3,5,4}][v_{1,3,6},v_{3,6,2}][v_{1,4,5},v_{4,5,2}][v_{1,4,6},v_{4,6,5}][v_{1,5,6},v_{5,6,3}]+&\\
& [v_{1,2,3},v_{2,3,4}] [v_{1,2,4},v_{2,4,6}] [v_{1,2,5},v_{2,5,3}][v_{1,2,6},v_{2,6,5}][v_{1,3,4},v_{3,4,5}][v_{1,3,5},v_{3,5,6}][v_{1,3,6},v_{3,6,2}][v_{1,4,5},v_{4,5,2}][v_{1,4,6},v_{4,6,3}][v_{1,5,6},v_{5,6,4}]-&\\
& [v_{1,2,3},v_{2,3,5}] [v_{1,2,4},v_{2,4,6}] [v_{1,2,5},v_{2,5,4}][v_{1,2,6},v_{2,6,3}][v_{1,3,4},v_{3,4,2}][v_{1,3,5},v_{3,5,4}][v_{1,3,6},v_{3,6,5}][v_{1,4,5},v_{4,5,6}][v_{1,4,6},v_{4,6,3}][v_{1,5,6},v_{5,6,2}]-&\\
& [v_{1,2,3},v_{2,3,6}] [v_{1,2,4},v_{2,4,5}] [v_{1,2,5},v_{2,5,3}][v_{1,2,6},v_{2,6,4}][v_{1,3,4},v_{3,4,2}][v_{1,3,5},v_{3,5,6}][v_{1,3,6},v_{3,6,4}][v_{1,4,5},v_{4,5,3}][v_{1,4,6},v_{4,6,5}][v_{1,5,6},v_{5,6,2}]-&\\
& [v_{1,2,3},v_{2,3,4}] [v_{1,2,4},v_{2,4,5}] [v_{1,2,5},v_{2,5,6}][v_{1,2,6},v_{2,6,3}][v_{1,3,4},v_{3,4,5}][v_{1,3,5},v_{3,5,2}][v_{1,3,6},v_{3,6,4}][v_{1,4,5},v_{4,5,6}][v_{1,4,6},v_{4,6,2}][v_{1,5,6},v_{5,6,3}]-&\\
& [v_{1,2,3},v_{2,3,6}] [v_{1,2,4},v_{2,4,3}] [v_{1,2,5},v_{2,5,4}][v_{1,2,6},v_{2,6,5}][v_{1,3,4},v_{3,4,6}][v_{1,3,5},v_{3,5,2}][v_{1,3,6},v_{3,6,5}][v_{1,4,5},v_{4,5,3}][v_{1,4,6},v_{4,6,2}][v_{1,5,6},v_{5,6,4}]-&\\
& [v_{1,2,3},v_{2,3,5}] [v_{1,2,4},v_{2,4,3}] [v_{1,2,5},v_{2,5,6}][v_{1,2,6},v_{2,6,4}][v_{1,3,4},v_{3,4,5}][v_{1,3,5},v_{3,5,6}][v_{1,3,6},v_{3,6,2}][v_{1,4,5},v_{4,5,2}][v_{1,4,6},v_{4,6,3}][v_{1,5,6},v_{5,6,4}]-&\\
& [v_{1,2,3},v_{2,3,4}] [v_{1,2,4},v_{2,4,6}] [v_{1,2,5},v_{2,5,3}][v_{1,2,6},v_{2,6,5}][v_{1,3,4},v_{3,4,6}][v_{1,3,5},v_{3,5,4}][v_{1,3,6},v_{3,6,2}][v_{1,4,5},v_{4,5,2}][v_{1,4,6},v_{4,6,5}][v_{1,5,6},v_{5,6,3}]\; \; \,\\
\end{eqnarray*}

\end{proposition}

\begin{proof}
We denote by $B(\otimes_{1\leq i<j<k\leq 6} (v_{i,j,k}))$ the right hand side of the above equality. In order to prove $B=det^{S^3}$ one can use the universality property of the $det^{S^3}$ map. It is easy to check that $B(\omega_{\mathcal{P}^{(1)}})=1$.

Next, we want to show that if $\otimes_{1\leq i<j<k\leq 6}(v_{i,j,k})\in V_2^{\otimes 20}$ such there exist $1\leq x< y<z<t\leq 6$ with the property that $v_{x,y,z}=v_{x,y,t}=v_{x,z,t}=v_{y,z,t}$ then $B(\otimes_{1\leq i<j<k\leq 6}(v_{i,j,k}))=0$. 

First notice that $B$ is invariant under the action of the subgroup $G\subseteq S_6$ that is generated by $(2,3)$, $(3,4)$, $(4,5)$ and $(5,6)$. Because of this fact it is enough to check the universality property for $(x,y,z,t)=(1,2,3,4)$, and for $(x,y,z,t)=(2,3,4,5)$. 

Case I: If $(x,y,z,t)=(1,2,3,4)$ then $v_{1,2,3}=v_{1,2,4}=v_{1,3,4}=v_{2,3,4}$. It is easy to see that all the terms in $B(\otimes_{1\leq i<j<k\leq 6}(v_{i,j,k}))$ are equal to $0$, for example the first term is equal to zero because $[v_{1,2,3},v_{2,3,4}]=0$. 

Case II: If $(x,y,z,t)=(2,3,4,5)$ then $v_{2,3,4}=v_{2,3,5}=v_{2,4,5}=v_{3,4,5}$. Let $v=v_{2,3,4}=v_{2,3,5}=v_{2,4,5}=v_{3,4,5}$. In this situations all the terms  will cancel in pairs. For example, if we look to the first term and thirty-seventh term we have the following expressions, respectively:
\[[v_{1,2,3},v] [v_{1,2,4},v] [v_{1,2,5},v][v_{1,2,6},v_{2,6,5}][v_{1,3,4},v_{3,4,6}][v_{1,3,5},v][v_{1,3,6},v_{3,6,2}][v_{1,4,5},v_{4,5,6}][v_{1,4,6},v_{4,6,2}][v_{1,5,6},v_{5,6,3}]\]
and
\[-[v_{1,2,3},v] [v_{1,2,4},v] [v_{1,2,5},v][v_{1,2,6},v_{2,6,5}][v_{1,3,4},v_{3,4,6}][v_{1,3,5},v][v_{1,3,6},v_{3,6,2}][v_{1,4,5},v_{4,5,6}][v_{1,4,6},v_{4,6,2}][v_{1,5,6},v_{5,6,3}]\]
So, we get these terms to sum up to zero. The rest of the matching pairs are given in the following table:
\begin{center}
\begin{tabular}{|c|c|}
 \hline
 Term&Matching Term\\
 \hline
 1&37\\
 \hline
 2&43\\
 \hline
 3&31\\
 \hline
 4&44\\
 \hline
 5&32\\
 \hline
 6&38\\
 \hline
 7&47\\
 \hline
 8&41\\
 \hline
 9&26\\
 \hline
 10&55\\
 \hline
 11&25\\
 \hline
 12&56\\
 \hline
 13&45\\
 \hline
 14&35\\
 \hline
 15&28\\
 \hline
 \end{tabular}
 \quad\quad\quad
 \begin{tabular}{|c|c|}
 \hline
 Term&Matching Term\\
 \hline
 16&57\\
 \hline
 17&27\\
 \hline
 18&58\\
 \hline
 19&39\\
 \hline
 20&33\\
 \hline
 21&30\\
 \hline
 22&59\\
 \hline
 23&29\\
 \hline
 24&60\\
 \hline
 25&11\\
 \hline
 26&9\\
 \hline
 27&17\\
 \hline
 28&15\\
 \hline
 29&23\\
 \hline
 30&21\\
 \hline
 \end{tabular}
 \quad\quad\quad
 \begin{tabular}{|c|c|}
 \hline
 Term&Matching Term\\
 \hline
 31&3\\
 \hline
 32&5\\
 \hline
 33&20\\
 \hline
 34&51\\
 \hline
 35&14\\
 \hline
 36&53\\
 \hline
 37&1\\
 \hline
 38&6\\
 \hline
 39&19\\
 \hline
 40&49\\
 \hline
 41&8\\
 \hline
 42&54\\
 \hline
 43&2\\
 \hline
 44&4\\
 \hline
 45&13\\
 \hline
 \end{tabular}
 \quad\quad\quad
 \begin{tabular}{|c|c|}
 \hline
 Term&Matching Term\\
 \hline
 46&50\\
 \hline
 47&7\\
 \hline
 48&52\\
 \hline
 49&40\\
 \hline
 50&46\\
 \hline
 51&34\\
 \hline
 52&48\\
 \hline
 53&36\\
 \hline
 54&42\\
 \hline
 55&10\\
 \hline
 56&12\\
 \hline
 57&16\\
 \hline
 58&18\\
 \hline
 59&22\\
 \hline
 60&24\\
 \hline
\end{tabular}
\end{center}

One can see that the pairs of terms have opposite signs. So, all the terms will cancel in pairs, which proves our statement. 

\end{proof}

\end{landscape}

\section{Equivalence Classes of $2$-partitions under the $S_6\times S_2$ Action}

As discussed earlier in the paper there are \num{184756}  homogeneous $2$-partitions of the hypergraph $K_6^3$. Using MATLAB, one can show that \num{13644} of them are  non-trivial (see Lemma \ref{trivpar}). We also know that there is a action of the group $S_6\times S_2$ on $\mathcal{P}_2^{h,nt}(K_6^3)$. In this section we present the $20$ equivalence classes of $\mathcal{P}_2^{h,nt}(K_6^3)$ under the action of $S_6\times S_2$ and the value of $\varepsilon^{S^3}$ on each element in $\mathcal{P}_2^{h,nt}(K_6^3)$. All results in this section were obtained using MATLAB. 

A summary of this information is presented in the following table, where $\mathcal{P}^{(i)}$ for $1\leq i\leq20$ are representatives of the $20$ equivalence classes:
\begin{center}
 \begin{tabular}{|c|c|c|}
  \hline
  Partition&$\epsilon^{S^3}(\mathcal{P}^{(i)})$&Orbit Size\\
  \hline
    $\mathcal{P}^{(1)}$&1&1440\\
    \hline
    $\mathcal{P}^{(2)}$&-1&240\\
    \hline
    $\mathcal{P}^{(3)}$&-1&1440\\
    \hline
    $\mathcal{P}^{(4)}$&1&360\\
    \hline
    $\mathcal{P}^{(5)}$&-1&1440\\
    \hline
    $\mathcal{P}^{(6)}$&1&1440\\
    \hline
    $\mathcal{P}^{(7)}$&-1&720\\
    \hline
    $\mathcal{P}^{(8)}$&-1&1440\\
    \hline
    $\mathcal{P}^{(9)}$&1&1440\\
    \hline
    $\mathcal{P}^{(10)}$&1&360\\
    \hline
    \end{tabular}
    \quad\quad\quad\quad
    \begin{tabular}{|c|c|c|}
    \hline
    Partition&$\epsilon^{S^3}(\mathcal{P}^{(i)})$&Orbit Size\\
    \hline
    $\mathcal{P}^{(11)}$&-1&720\\
    \hline
    $\mathcal{P}^{(12)}$&1&720\\
    \hline
    $\mathcal{P}^{(13)}$&1&360\\
    \hline
    $\mathcal{P}^{(14)}$&-1&360\\
    \hline
    $\mathcal{P}^{(15)}$&-1&72\\
    \hline
    $\mathcal{P}^{(16)}$&1&360\\
    \hline
    $\mathcal{P}^{(17)}$&1&120\\
    \hline
    $\mathcal{P}^{(18)}$&-1&360\\
    \hline
    $\mathcal{P}^{(19)}$&1&240\\
    \hline
    $\mathcal{P}^{(20)}$&-4&12\\
    \hline
 \end{tabular}
\end{center}

In  Figure \ref{RepStart} to Figure \ref{RepEnd} we present $20$ partitions that are representatives  for the equivalences classes under the action of $S_6\times S_2$. To understand these picture better, we consider Figure \ref{RepStart} and give an explicit description of the hyperedges  in $\mathcal{P}^{(1)}$. 
Here we have 
$$\tiny{E(\mathcal{H}^{(1)}_1)=\{\{1,2,3\}, \{1,2,4\},\{1,3,4\},\{1,2,5\}, \{3,4,5\},\{1,5,6\},\{2,4,6\},\{2,5,6\},\{3,4,6\},\{4,5,6\}\}}$$ and  $$\tiny{E(\mathcal{H}^{(1)}_2)=\{\{1,2,6\},\{1,3,5\}, \{1,3,6\},\{1,4,5\},\{1,4,6\}, \{2,3,4\},\{2,3,5\},\{2,3,6\},\{2,4,5\},\{3,5,6\}}\}$$
given by which triangles in the picture are shaded and not shaded, respectively.
The other representatives can be obtained from the corresponding pictures similarly.

%%%%%%%%%%%%%%%%%
% Rep 1
\begin{figure}[htbp]
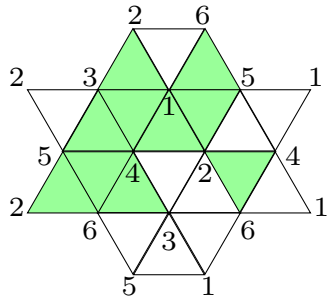

 \centering

 \caption{$\mathcal{P}^{(1)}$ of orbit size 1440 and sign $\epsilon^{S^3}(\mathcal{P}^{(1)})=1$}
 \label{RepStart}
\end{figure}
%%%%%%%%%%%%%%%%%%%%%%%%%
% Rep 2
\begin{figure}[htbp]
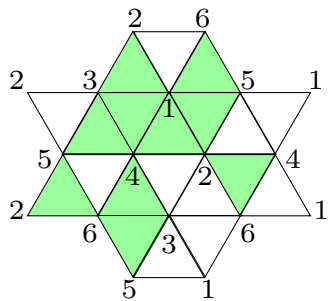

 \centering
 \vspace*{5em}
 %
 \caption{$\mathcal{P}^{(2)}$ of orbit size 240 and sign $\epsilon^{S^3}(\mathcal{P}^{(2)})=-1$}
\end{figure}
%%%%%%%%%%%%%%%%%%%%%%%
% Rep 3
\begin{figure}[htbp]
 \centering
 %
 \caption{$\mathcal{P}^{(3)}$ of orbit size 1440 and sign $\epsilon^{S^3}(\mathcal{P}^{(3)})=-1$}
\end{figure}
%%%%%%%%%%%%%%%%%%%%%%%%%
% Rep 4
\begin{figure}[htbp]
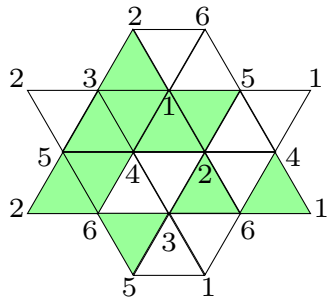

 \centering
 %
 \caption{$\mathcal{P}^{(4)}$ of orbit size 360 and sign $\epsilon^{S^3}(\mathcal{P}^{(4)})=1$}
\end{figure}
%%%%%%%%%%%%%%%%%%%
% Rep 5
\begin{figure}[htbp]
 \centering
 %
 \caption{$\mathcal{P}^{(5)}$ of orbit size 1440 and sign $\epsilon^{S^3}(\mathcal{P}^{(5)})=-1$}
\end{figure}
%%%%%%%%%%%%%%%%%%%%%%%%%%%%%
% Rep 6
\begin{figure}[htbp]
 \centering
 %
 \caption{$\mathcal{P}^{(6)}$ of orbit size 1440 and sign $\epsilon^{S^3}(\mathcal{P}^{(6)})=1$}
\end{figure}
%%%%%%%%%%%%%%%%%%%%%%%%%%%%
% Rep 7
\begin{figure}[htbp]
 \centering
 %
 \caption{$\mathcal{P}^{(7)}$ of orbit size 720 sign $\epsilon^{S^3}(\mathcal{P}^{(7)})=-1$}
\end{figure}
%%%%%%%%%%%%%%%%%%%
% Rep 8
\begin{figure}[htbp]
 \centering
 %
 \caption{$\mathcal{P}^{(8)}$ of orbit size 1440 and sign $\epsilon^{S^3}(\mathcal{P}^{(8)})=-1$}
\end{figure}
%%%%%%%%%%%%%%%%%%%%%%%%%%%%
% Rep 9
\begin{figure}[htbp]
 \centering
 %
 \caption{$\mathcal{P}^{(9)}$ of orbit size 1440 and sign $\epsilon^{S^3}(\mathcal{P}^{(9)})=1$}
\end{figure}
%%%%%%%%%%%%%%%%%%%%%%%%%
% Rep 10
\begin{figure}[htbp]
 \centering
 %
 \caption{$\mathcal{P}^{(10)}$ of orbit size 360 and sign $\epsilon^{S^3}(\mathcal{P}^{(10)})=1$}
\end{figure}
%%%%%%%%%%%%%%%%%%%%
% Rep 11
\begin{figure}[htbp]
 \centering
 %
 \caption{$\mathcal{P}^{(11)}$ of orbit size 720 and sign $\epsilon^{S^3}(\mathcal{P}^{(11)})=-1$}
\end{figure}
%%%%%%%%%%%%%%%%%%%%%%
% Rep 12
\begin{figure}[htbp]
 \centering
 %
 \caption{$\mathcal{P}^{(12)}$ of orbit size 720 and sign $\epsilon^{S^3}(\mathcal{P}^{(12)})=1$}
\end{figure}
%%%%%%%%%%%%%%%%%%%%%%
% Rep 13
\begin{figure}[htbp]
 \centering
 %
 \caption{$\mathcal{P}^{(13)}$ of orbit size 360 and sign $\epsilon^{S^3}(\mathcal{P}^{(13)})=1$}
\end{figure}
%%%%%%%%%%%%%%%%%%%%%%%%%%%%
% Rep 14
\begin{figure}[htbp]
 \centering
 %
 \caption{$\mathcal{P}^{(14)}$ of orbit size 360 and sign $\epsilon^{S^3}(\mathcal{P}^{(14)})=-1$}
\end{figure}
%%%%%%%%%%%%%%%%%%%%%%%%%%%
% Rep 15
\begin{figure}[htbp]
 \centering
 %
 \caption{$\mathcal{P}^{(15)}$ of orbit size 72 and sign $\epsilon^{S^3}(\mathcal{P}^{(15)})=-1$}
\end{figure}
%%%%%%%%%%%%%%%%%%%%%%%%%%
% Rep 16
\begin{figure}[htbp]
 \centering
 %
 \caption{$\mathcal{P}^{(16)}$ of orbit size 360 and sign $\epsilon^{S^3}(\mathcal{P}^{(16)})=1$}
\end{figure}
%%%%%%%%%%%%%%%%%%%%%%%%
% Rep 17
\begin{figure}[htbp]
 \centering
 %
 \caption{$\mathcal{P}^{(17)}$ of orbit size 120 and sign $\epsilon^{S^3}(\mathcal{P}^{(17
 )})=1$}
\end{figure}
%%%%%%%%%%%%%%%%%%%%%%
% Rep 18
\begin{figure}[htbp]
 \centering
 %
 \caption{$\mathcal{P}^{(18)}$ of orbit size 360 and sign $\epsilon^{S^3}(\mathcal{P}^{(18)})=-1$}
\end{figure}
%%%%%%%%%%%%%%%%%%%%%%%
% Rep 19
\begin{figure}[htbp]
 \centering
 %
 \caption{$\mathcal{P}^{(19)}$ of orbit size 240 and sign $\epsilon^{S^3}(\mathcal{P}^{(19)})=1$}
\end{figure}
%%%%%%%%%%%%%%%%%%%%%%%%%%
% Rep 20
\begin{figure}[h]
 \centering
 %
 \caption{$\mathcal{P}^{(20)}$ of orbit size 12 and sign $\epsilon^{S^3}(\mathcal{P}^{(20)})=-4$}
 \label{RepEnd}
\end{figure}

\clearpage
\bibliographystyle{amsalpha}

\end{document}